\documentclass[11pt]{amsart}

\usepackage{amsmath,amsthm}
\usepackage{mathtools}
\usepackage{subfig}
\usepackage{amssymb}
\usepackage{graphicx}
\usepackage{nicematrix}
\usepackage{pinlabel}
\usepackage{lineno}
\usepackage{fullpage}
\usepackage{hyperref}

\newcommand{\ints}{\mathbb{Z}}
\newcommand{\reals}{\mathbb{R}}

\newcommand{\bG}{\mathbb{G}}

\newcommand{\bB}{\mathbb{B}}
\newcommand{\bT}{\mathbb{T}}
\newcommand{\embG}{\mathbf{G}}

\DeclareMathOperator{\im}{im}
\DeclareMathOperator{\diag}{diag}
\DeclareMathOperator{\Adj}{\Theta}

\newtheorem{theorem}{Theorem}[section]
\newtheorem{corollary}[theorem]{Corollary}

\newtheorem{lemma}[theorem]{Lemma}
\newtheorem{proposition}[theorem]{Proposition}

\theoremstyle{definition}
\newtheorem{convention}[theorem]{Convention}
\newtheorem{definition}[theorem]{Definition}
\newtheorem{example}[theorem]{Example}

\theoremstyle{remark}
\newtheorem{remark}[theorem]{Remark}

\title{The critical group of a combinatorial map}
\author{Criel Merino${}^{*}$}
\address{Instituto de Matem\'aticas, Universidad Nacional Aut\'onoma de M\'exico, Ciudad de M\'exico, 04510, M\'exico}
\email{merino@matem.unam.mx}
\thanks{${}^{*}$~Supported by Grant CONACYT A1-S-8195.}

\author{Iain Moffatt${}^{\dagger}$}
\address{Department of Mathematics, Royal Holloway, University of London, Egham, TW20~0EX, United Kingdom}
\email{iain.moffatt@rhul.ac.uk}
\thanks{${}^{\dagger}$~Supported the Engineering and Physical Sciences Research Council [grant number EP/W033038/1].}

\author{Steven Noble${}^{\ddagger}$}
\address{School of Computer Science, University~of~Leeds, Leeds, LS2~9JT, United Kingdom}
\email{s.d.noble@leeds.ac.uk}
\thanks{${}^{\ddagger}$~Supported the Engineering and Physical Sciences Research Council [grant number EP/W032945/1].
\\
For the purpose of open access, the authors have applied a Creative Commons
Attribution (CC BY) licence to any Author Accepted Manuscript version arising. \\ No underlying data is associated with this article.}

\subjclass[2020]{Primary 05C10, 05C25; Secondary 05C50, 05C57, 20K01, 91A43}

\date{March 17, 2025}

\keywords{chip firing, critical group, embedded graph, sandpile group, sandpile model, Laplacian, map, Matrix–Tree Theorem, quasi-tree}

\begin{document}

\begin{abstract}
Motivated by the appearance of embeddings in the theory of chip-firing and the critical group  of a graph, we introduce a version of the critical group (or sandpile group) for combinatorial maps, that is, for  graphs embedded in orientable surfaces. 
We provide several definitions of our critical group, by approaching it through analogues of the cycle--cocycle matrix, the Laplacian matrix, and as the group of critical states of a chip-firing game (or sandpile model) on the edges of a map.

Our group can be regarded as a perturbation of the classical critical group of its underlying graph by topological information, and it agrees with the classical critical group in the plane case. 
Its cardinality is equal to the number of spanning quasi-trees in a connected map, just as the cardinality of the classical critical group is equal to the number of spanning trees of a connected graph.

Our approach exploits the properties of principally unimodular matrices and the methods of delta-matroid theory.
\end{abstract}

\maketitle

\section{Introduction}\label{sec:intro}

The \emph{critical group} $K(G)$ is a finite Abelian group associated with a connected graph $G$. 
Its order is equal to the number of spanning trees in the graph, a fact equivalent to Kirchhoff's Matrix--Tree Theorem and a consequence of its definition as the cokernel of the reduced graph Laplacian. 
The  group is well-established in combinatorics and physics arising in, for example, the theory of the chip-firing game (known as the sandpile model in physics), parking functions, flow spaces and cut spaces. It has connections with algebraic and arithmetic geometry as a Picard group and through the Riemann--Roch Theorem~\cite{MR2355607}.  As often happens with objects arising in a variety of different areas, the critical group goes by several different names, including the \emph{sandpile group}~\cite{MR1044086}, the \emph{group of components}~\cite{MR1129991}, the critical group~\cite{MR1676732}, or the \emph{Jacobian} of a graph~\cite{zbMATH01106861}. An overview of the historical development of the critical group is contained in~\cite{zbMATH05379217}, and we refer the reader to the books~\cite{MR3793659,zbMATH06952170} for general background on it.

Although the critical group is defined as a purely graph theoretic notion, depending entirely on the abstract structure of a graph, many aspects of the area are  topological in nature. 
Planarity appears in early  results on the critical group due to Berman~\cite{MR819700} and Cory and Rossin~\cite{MR1756151}, and more recent results of Chan, Thomas and Grochow~\cite{MR3373049}.
Results such as those of Baker and Wang~\cite{BW}, Ding~\cite{MR4328899} 
and Shokrieh and Wright~\cite{MR4643471} on Bernardi and rotor-router torsor structures hold only in the plane case. 
Indeed, although both torsors are defined on abstract graphs their
definition requires a cyclic ordering of the edges around each vertex provided by an embedding.

We believe results that hold only for plane and planar graphs should be instances of more general results about graphs embedded in higher genus surfaces. Therefore, in order to properly understand key aspects of the critical
group, the natural setting should be topological graph theory. So in order
to achieve the full potential of the theory one must extend the definition of the critical group so that it no longer depends only on the adjacencies of the graph but instead takes into account an embedding.

In this paper we define the critical group for embedded graphs. 
As concepts underlying the various definitions of the classical critical group do not directly extend to the topological case,
 new ideas are needed.
We are guided by recent work~\cite{zbMATH07094555,zbMATH07055741,zbMATH07456251}  showing that the analogous object to a matroid in the topological setting is a delta-matroid. (For this paper, no knowledge of matroids or delta-matroids is
required.) 
Our starting point is the definition of the classical critical group
using fundamental cycles and cocycles, which are objects associated with the cycle matroid of a graph. 
Neither the notion of a fundamental cycle nor a fundamental cocycle extends immediately to the topological setting. 
A key technical contribution of this paper is the resolution of these difficulties by drawing on methods from delta-matroid theory.
By applying these techniques, we show that
the collection of fundamental cycles and cocycles taken together may be extended to embedded graphs, 
leading to a natural analogue of the cycle--cocycle matrix.
We define the critical group of an embedded graph to be the cokernel of this matrix.

It is not immediately obvious that the group is well-defined. Again, we are able to exploit techniques from delta-matroid theory to show that this is indeed the case.
In general, the theory of the critical group of maps is distinct from that for abstract graphs, with the theories coinciding if and only if the map is plane. 
Nevertheless, the theory for maps parallels that for abstract graphs.  
For example, a key consequence of this new theory is the statement and proof of a version of Kirchhoff's Matrix--Tree Theorem for graphs in surfaces, where the number of spanning trees is replaced by the number of spanning quasi-trees in an embedded graph.

In analogy with the classical theory, we also give a definition for the critical group of an embedded graph in terms of a Laplacian. The proof that this Laplacian definition coincides with the one described above is a key technical result of the paper.   
In particular, we apply it to show that the critical group of a map can be defined in terms of the critical states of a chip-firing game.  
The bullet points in Section~\ref{sec:bullets} provide an overview of all of our main results.

We note that while our approach is informed throughout by delta-matroids, our final solution does not reference delta-matroids explicitly. Informally speaking, the amenability of the critical group of an embedded graph to techniques from delta-matroid theory suggests that it should depend solely on the delta-matroid of the embedded graph. However, somewhat surprisingly,  we show in Remark~\ref{dfth} that this is not the case. Nevertheless, as we demonstrate, the methods prove to be invaluable. 

\medskip

This paper is structured as follows. In Section~\ref{sec:prelim}, we provide key background definitions, discuss the critical group in the classical case, and review (combinatorial) maps which provide our formalism for describing graphs embedded in surfaces.
In Section~\ref{sec:bullets} we introduce the critical group of a map and provide an overview of our main results. We then prove that the group is well-defined in Section~\ref{s:well} before establishing connections with the classical case in Section~\ref{sec:connections}. In particular, we describe how the critical group of a map can be regarded as a perturbation of the classical critical group of its underlying graph by topological information.  
Section~\ref{sec:qtrees} provides an analogue of Kirchhoff's Matrix--Tree Theorem for embedded graphs, and what the critical group tells us about  spanning quasi-trees.  
In Section~\ref{s.Laplace},  we give an analogue of the Laplacian definition for the critical group, proving that the two definitions result in isomorphic groups.
We exploit the Laplacian definition in Section~\ref{sec:chip} to describe a chip-firing game on maps for which the critical group arises as its group of critical states. 
To illustrate key points in the theory we close by computing examples of the critical group and comparing the results with the critical groups of the corresponding non-embedded graphs.

\medskip

We would like to thank Petter Br\"and\'en, Tom Tucker, and Roman Nedela for very helpful discussions. A revised version of this paper was prepared while I.M and S.N. were participating in an INI Retreat and we would like to thank the Isaac Newton Institute for their hospitality. 

\section{Preliminaries}\label{sec:prelim}
\subsection{The classical critical group}\label{ss:class}

An \emph{abstract graph} comprises two sets $V$ and $E$, its sets of \emph{vertices} and \emph{edges}, respectively. Two not necessarily distinct vertices, the \emph{endvertices} of $e$, are associated with each edge $e$. Abstract graphs are normally just called graphs, but we prepend the adjective abstract in order to clearly distinguish them from  embedded graphs which will be our main topic of interest. Observe that abstract graphs may have \emph{multiple edges}, that is edges with the same set of endvertices, and \emph{loops} which are edges whose two endvertices coincide. We assume familiarity with the basic concepts of graph theory as in~\cite{zbMATH05202336}.

If each edge of an abstract graph $G$ is given a direction, then we obtain 
an abstract directed graph $\vec G$. 
Let $\vec G$ be a directed, connected abstract graph with edge set $E$, and $T$ be the edge set of a spanning tree of $G$. 
For each edge $e$ in $E- T$, the \emph{fundamental cycle}, $C(T,e)$ is the unique cycle of $G$ whose edges all belong to $T\cup \{e\}$. 
Its \emph{signed incidence vector} $x$ is a vector indexed by $E$
and satisfies $x_i=0$ if $i \notin C(T,e)$ and $x_i=\pm 1$ if $i \in C(T,e)$, with the sign determined as follows. If we walk around $C(T,e)$ in the direction determined by the direction of $e$ in $\vec G$, then $x_i=1$ if and only if the direction of edge $i$ in $\vec G$ is consistent with the direction of our walk. 
For each edge $e$ in $T$, the \emph{fundamental cocycle}, $C^{\ast}(T,e)$ is the unique cocycle of $G$, that is, a minimal set of edges whose removal disconnects $G$, containing no edge in $T$ except $e$. 
Its \emph{signed incidence vector} $x$ satisfies $x_i=0$ if $i \notin C^{\ast}(T,e)$, $x_i=1$ if $i \in C^{\ast}(T,e)$ and
the head of $i$ lies in the same component of $G\backslash C^{\ast}(T,e)$ as the head of $e$ and 
$x_i=-1$ if $i \in C^{\ast}(T,e)$ and
the head of $i$ lies in the same component of $G\backslash C^{\ast}(T,e)$ as the tail of $e$. The \emph{signed cycle--cocycle matrix} $M(\vec G,T)$ is the matrix with rows and columns indexed by $E$, so that the row indexed by edge $e$ is the signed incidence vector of either the fundamental directed cycle containing $e$ or the fundamental directed cocycle containing $e$ depending on whether or not $e$ is an edge in $T$. See~\cite{MR2849819}  for background.

If we regard the rows of an $m\times m$ integer matrix $A$ as elements of the free $\ints$-module $\ints^m$, then the quotient module $\ints^m/ \langle A \rangle$ is called the \emph{cokernel} of $A$. The \emph{(classical) critical group} $K(G)$ of a connected abstract graph $G$ can be defined as the cokernel of the
signed cycle--cocycle matrix $M(\vec G,T)$:
\[K(G) := \ints^{|E|} / \langle M(\vec G,T) \rangle.\]
It is independent of the choice of direction given to the edges of $G$ and the choice of $T$. See~\cite{MR1676732} for details.

\medskip

We now describe an alternative approach to defining the classical critical group of a connected, abstract graph. Suppose that $G$ is connected and loopless.
Then its Laplacian, $\Delta(G)$, is the matrix with rows and columns indexed by the vertices of $G$ defined by 
\[ \Delta(G)_{i,j} := \begin{cases} -m_{i,j} & \text{if $i\neq j$,}\\
\deg(i) & \text{if $i= j$.} \end{cases}\]
Here, $\deg(i)$ is  the degree of vertex $i$, that is the number of times it occurs as an edge end; and $m_{i,j}$ is the number of edges joining $i$ and $j$.

Any matrix obtained from $\Delta(G)$ by choosing a vertex $q$ of $G$ and deleting the corresponding row and column of $\Delta(G)$ is called a \emph{reduced Laplacian} of $G$ and is denoted by $\Delta^q(G)$. The celebrated Matrix--Tree Theorem~\cite{18471481202} is as follows.
\begin{theorem}[Matrix--Tree Theorem]\label{kmt}
Let $G$ be a connected, loopless abstract graph with vertex $q$. Then the number of spanning trees of $G$ is equal to $\det(\Delta^q(G))$.
\end{theorem}

The following theorem is well-known and gives an alternative way of defining the classical critical group. A proof can be found in the book~\cite{Chris+Gordon}. 
\begin{theorem}\label{thm:classicalcritgroups}
Let $G=(V,E)$ be a connected, loopless abstract graph and let $\vec G$ be obtained from $G$ by giving a direction to each of its edges. Let $T$ be the edge set of a spanning tree of $G$ and $q$ be a vertex of $G$. Then
\[   \ints^{|V|-1} / \langle \Delta^q(G) \rangle \cong \ints^{|E|} / \langle M(\vec G,T) \rangle.\] 
\end{theorem}
Thus we may also define the critical group using the reduced Laplacian. The restriction that $G$ be loopless is not significant, because if $G'$ is obtained from $G$ by deleting a loop, then the critical groups of $G$ and $G'$ are isomorphic. The critical group of a disconnected graph is defined to be
 the direct sum of the critical groups of its connected components.

Applying the Matrix--Tree Theorem and Theorem~\ref{thm:classicalcritgroups} gives the following important property of the group.
\begin{theorem}\label{classicalorder}
Let $G$ be a connected, loopless, abstract graph. Then the order of its critical group $K(G)$ is the number of spanning trees of $G$.
\end{theorem}

\subsection{Combinatorial maps}

We are interested in \emph{embedded graphs} or \emph{topological maps} which here are graphs embedded in oriented surfaces.  These can be thought of as graphs drawn on a closed oriented surface in such a way that their edges do not intersect (except where edge ends share a  vertex).
 All our embeddings are cellular, meaning that faces are homeomorphic to discs, and generally  our graphs are  connected (although our results   extend easily to  disconnected graphs).  
 The \emph{genus} of an embedded graph is the genus of the surface it is embedded in. An embedded graph is \emph{plane} if it has genus zero.

We shall realise embedded graphs here as combinatorial maps. (Of course, everything we do here can be formulated in terms of any description of an embedded graph.) A \emph{combinatorial map}, whose name we shorten here to just \emph{map}, consists of a finite set $D$ whose elements are called \emph{darts} and a triple of permutations $[\sigma, \alpha, \phi]$ on $D$ such that (i) $\alpha$ is a fixed-point-free involution, and (ii) $\phi=\sigma^{-1}\alpha^{-1}$.  
Any singleton cycles in $\sigma$ and $\phi$ are not suppressed, so that their common domain is explicitly the set of darts. 
We shall generally write a map as just a triple $[\sigma, \alpha, \phi]$ omitting the set $D$ of darts, which is implicit, from the notation. 
The cycles in $\sigma$ are called the \emph{vertices} of the map, the cycles in $\alpha$ are the \emph{edges}, and the cycles in $\phi$ are the \emph{faces}. 
The \emph{genus}, denoted by $g(\bG)$, of a connected map $\bG$ is given by 
\begin{equation}\label{eqn:genus} g(\bG)=\tfrac{1}{2}(2-v+e-f),\end{equation} where $v$, $e$ and $f$ are the numbers of vertices, edges and faces of $\bG$. A connected map is \emph{plane} if it has genus zero. 

We say that a map $[\sigma,\alpha,\phi]$ is a \emph{bouquet} if it has only one vertex or equivalently if $\sigma$ is cyclic, and that it is a \emph{quasi-tree} if it has only one face or equivalently if $\phi$ is cyclic.
 A map is \emph{connected} if whenever we partition the cycles of $\sigma$ (or $\phi$) into two non-empty, disjoint collections $C_1$ and $C_2$ then there is an edge, equivalently a cycle of $\alpha$, with one dart belonging to a cycle in $C_1$ and the other belonging to a cycle in $C_2$.

We need to allow the possibility that the set of darts is empty. In this case each of $\sigma$, $\alpha$ and $\phi$ is defined to be the identity permutation with domain equal to the empty set. We stipulate that $\sigma$ and $\phi$ both comprise one empty cycle and that $\alpha$ is the empty permutation with no cycles. (This is to ensure that, in what is to follow, the number of cycles of $\sigma$, $\alpha$ and $\phi$ correctly record the number of vertices, edges and faces of an embedded graph consisting of one isolated vertex.)

\medskip

\begin{figure}
\begin{center}
  \labellist
\small\hair 2pt 
\pinlabel {$d$} at    62 34
\pinlabel {$\alpha(d)$} at  87 21 
\pinlabel {$\sigma(d)$} at  32 43
\pinlabel {$\phi(d)$} at  97 43
\endlabellist
 \includegraphics[scale=1.2]{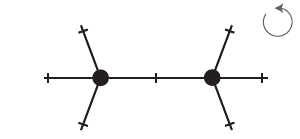} 
\end{center}
\caption{Moving between maps and embedded graphs.}
\label{fig:move}
\end{figure}

Maps and embedded graphs correspond. For convenience we provide a brief overview of this connection, referring the reader to the book~\cite{zbMATH02019766} for additional details. The reader may find it helpful to consult Figure~\ref{fig:move} and  Example~\ref{basicexamp1} which illustrate the concepts of this subsection. 
Given a connected embedded graph $\embG$ (recall the graph is in an oriented surface), we think of each edge as comprising two \emph{half-edges} or \emph{darts}. Let $D$ denote the set of darts.
In order to define the triple $[\sigma, \alpha, \phi]$, we imagine the name or label of each dart being placed adjacent to the dart in the face $f$ meeting the dart determined as follows. Suppose that dart $d$ meets the vertex $v$. 
Then the face $f$ is the first face encountered when rotating around $v$ in the direction consistent with the orientation of the surface   starting at $d$. 
Suppose now that $\embG$ has at least one edge.
The permutation $\alpha$ is the product of disjoint $2$-cycles, one for each edge of $\embG$, so that the $2$-cycle corresponding to an edge $e$ comprises the two darts associated with $e$. 
The permutation $\sigma$ is the product of disjoint cycles, one for each vertex of $\embG$. The cycle corresponding to $v$ contains the darts adjacent with $v$ in the cyclic order in which they are encountered when rotating around $v$ in the direction consistent with the orientation of the surface.  
The permutation $\phi$ is the product of disjoint cycles, one for each face of $\embG$. The cycle corresponding to $f$ contains the labels of darts lying just within $f$
in the cyclic order in which they are encountered when walking just inside the boundary of $f$ in the direction consistent with the orientation of the surface. Recall that any singleton cycles in $\sigma$ and $\phi$ are not suppressed, so the numbers of disjoint cycles of $\sigma$ and $\phi$ are the numbers of vertices and faces of $\embG$ respectively. It is easy to check that the resulting map is connected. 
The case when $\embG$ has no edges is special and is described by the map in which  $\sigma$, $\alpha$, and $\phi$ are all identity permutations. 
Notice that in either case $\phi=\sigma^{-1}\alpha^{-1}$. 
This process described is reversible  (see, for example,~\cite{zbMATH02019766} for details) and every connected map determines a connected embedded graph. 

We say that two maps $\bG=[\sigma, \alpha, \phi]$ and $\bG'=[\sigma', \alpha', \phi']$ are \emph{isomorphic} if there is a bijection $\beta$ from the set of darts of $\bG$ to the set of darts of $\bG'$, so that $\beta$ sends $\sigma$, $\alpha$ and $\phi$ to $\sigma'$, $\alpha'$ and $\phi'$, respectively. Then two connected embedded graphs are isomorphic if and only if their corresponding maps are. By Euler's Formula the definitions for the genus of a map and embedded graph agree. 

More generally, if $\embG$ is an embedded graph which is not necessarily connected, then 
each of the permutations $\sigma$, $\alpha$ and $\phi$ is found by determining its values on the connected components of $\embG$ and taking their composition.  
It is no longer the case that the process is completely reversible: the addition or removal of isolated vertices from $\embG$ makes no difference to any of the permutations\footnote{One could keep track of isolated vertices by allowing $\sigma$ and $\phi$ to have a number of empty cycles, one for each isolated vertex. For simplicity, other than the map with one vertex and no edges, we do not allow our maps to have isolated vertices.}
It is, however, possible to recover the components of $\embG$ that are not isolated vertices from its description as a map.

\medskip

Occasionally, we make use of rotation systems, as these provide a simple visual description of maps, suitable for figures. We now explain the straightforward equivalence with maps.

As $\phi=\sigma^{-1}\alpha^{-1}$, a  map is completely determined by $\sigma$ and $\alpha$. It follows that a  map, and therefore an embedded graph,  may also be described by  
a \emph{rotation system}, that is, a graph together with a cyclic order of the half-edges at each vertex: each cycle in $\sigma$ gives the vertices together with their cyclically ordered incident half-edges, and each cycle in $\alpha$ specifies which half-edges form an edge. Alternatively, a rotation system can be read from an embedded graph by equipping its underlying graph with the cyclic order of half-edges induced by the surface orientation.
Additional background on rotation systems can be found in the book~\cite{zbMATH01624430}.

For a map $\bG$, we define its \emph{underlying abstract graph} to be the abstract graph formed from $\bG$ by disregarding all the topological information. More precisely, if $\bG=[\sigma,\alpha,\phi]$, then the underlying abstract graph has one vertex for each cycle of $\sigma$ and one edge for each cycle of $\alpha$. The endpoints of an edge are the vertices corresponding to the cycles of $\sigma$ containing its constituent darts.

\medskip

Much of the time we will work with directed  maps. 
We note that to avoid confusion the term ``orientation'' is used in this paper  exclusively in the topological sense (e.g, a surface orientation) and never in the directed graph sense (i.e., a directed graph that arises by directing each edge of a graph).   
A \emph{directed map} consists of a map $[\sigma, \alpha, \phi]$ in which one element in each cycle of $\alpha$ is assigned to be a \emph{head} and the other element a \emph{tail}.
It is clear that directed maps correspond to directed embedded graphs and to directed rotation systems (each defined by assigning a direction to each edge of the graph to turn it into a directed graph) and the darts are assigned to be heads or tails in the obvious way. We can \emph{reverse} the direction of an edge by interchanging which darts in its cycle are assigned as the head and tail.
Every directed map $\vec{\bG}$ can be considered as a (non-directed) map $\bG$. We refer to such a $\bG$ as the \emph{underlying map} of $\vec{\bG}$, and say that $\vec{\bG}$ is obtained by \emph{directing the edges} of $\bG$.
 Every directed map also has an underlying abstract directed graph.

To simplify notation we shall use the following convention when working with directed maps.
\begin{convention}\label{con1}
In a directed map we assume the darts of the edge $e$ are $e^+$ and $e^-$ with $e^+$ labelling the head of the edge and $e^-$ labelling the tail.
\end{convention}
We extend this convention to directed embedded graphs and directed rotation systems.

\bigskip 

If $\embG$ is an embedded graph then its \emph{geometric dual} $\embG^*$ is the embedded graph whose vertices are obtained through the standard construction of placing one vertex in each face of $\embG$. For each edge $e$ of $\embG$, there is a corresponding edge in $\embG^*$ crossing $e$ exactly once and with ends on the dual vertices in the face or faces adjacent to $e$ in $\embG$. 
In terms of  maps, the geometric dual of $\bG=[\sigma,\alpha, \phi]$ is
\[\bG^*:= [\phi^{-1},\,\alpha,\, \sigma^{-1}]  = [\alpha \sigma,\, \alpha, \,\sigma^{-1}] .\] 
Again we refer the reader to~\cite{zbMATH02019766} for details.

We shall make use of S.~Chmutov's surprising construction of partial duals introduced in \cite{zbMATH05569114} and first described for maps in~\cite{pdhyper}. 
 Partial duality enables the formation of the geometric dual one edge at a time, or more generally the formation of the geometric dual with respect to only a subset of edges.   Additional background on partial duals can be found in the book~\cite{MR3086663}.

Given a  map $\bG=[\sigma,\alpha, \phi]$ and a set $A$ of its edges, let $\alpha^A$ denote the permutation so that $\alpha^A(d)=\alpha(d)$ if $d$ is a dart associated with an edge in $A$ and $\alpha^A(d)=d$ if $d$ is a dart that is associated with an edge that is not in $A$. With a slight abuse of terminology, $\alpha^A$ is the restriction of $\alpha$ to darts associated with edges in $A$. Then the \emph{partial dual} $\bG^A$ is defined to be 
\[  \bG^A:= [\alpha^A\sigma, \,\alpha,\, \sigma^{-1}\alpha^{A^c}],\]
where $A^c$ denotes the complement of $A$. 
Notice that  \[ (\alpha^A \sigma)^{-1} \alpha^{-1} = \sigma^{-1} (\alpha^A)^{-1} \alpha=\sigma^{-1}\alpha^{A^c},\] so $\bG^A$ is indeed a map. 
Also notice that the operation of forming a partial dual is an involution, and that $\bG^E = \bG^*$, where $E$ is the set of all edges of $\bG$.
The definition of a partial dual also applies to directed maps.

Figure~\ref{figpd} shows partial duality as a local operation on directed rotation systems. Here the partial dual $\vec{\bG}^{\{e\}}$ is obtained by modifying a directed rotation system $\vec{\bG}$ locally at the edge $e$: from left to right if $e$ is not a loop, from right to left if $e$ is a loop.  The partial dual $\vec{\bG}^A$ is obtained by applying this operation to each edge in $A$ in any order. (Described in this way it is not immediately obvious that the resulting directed rotation system is independent of this choice of order, but this is clear from the definition in terms of maps.)

\medskip

\begin{figure}
\begin{center}
\begin{tabular}{ccc}
  \labellist
\small\hair 2pt 
\pinlabel {$e^-$} at    42 42
\pinlabel {$e^+$} at    65 31
\pinlabel {$x$} at   9 47
\pinlabel {$y$} at   8 31
\pinlabel {$z$} at  17 22
\pinlabel {$a$} at  96 28
\pinlabel {$b$} at  99 41
\pinlabel {$c$} at  91 52
\endlabellist
 \includegraphics[scale=1.2]{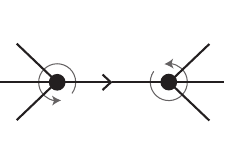} 
&
\quad\raisebox{14mm}{$\xleftrightarrow[\text{edge $e$}]{\text{~ partial dual ~}}$}\quad
&
  \labellist
\small\hair 2pt
\pinlabel {$e^-$} at    33 62
\pinlabel {$e^+$} at    43.5 19 
\pinlabel {$x$} at   19 46
\pinlabel {$y$} at   18 31
\pinlabel {$z$} at  27 22
\pinlabel {$a$} at  55 23
\pinlabel {$b$} at  57 41
\pinlabel {$c$} at  54 59
\endlabellist
 \includegraphics[scale=1.2]{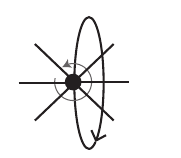} 
\end{tabular}

\end{center}
\caption{The partial dual of an edge in a directed rotation system.}
\label{figpd}
\end{figure}

Given a map $\bG=[\sigma,\alpha,\phi]$ with edge set $E$, we say that 
$\bG'=[\sigma',\alpha',\phi']$ is an \emph{induced submap} of $\bG$ if $\bG'$ is a map and there is a subset $A$ of $E$ so that the following hold:
\begin{enumerate}
\item if $A=\emptyset$, then both $\sigma'$ and $\phi'$ comprise a single empty cycle and $\alpha'$ has no cycles;
\item if $A\ne \emptyset$, then both $\sigma'$ and $\alpha'$ are obtained from $\sigma$ and $\alpha$ respectively by deleting all the darts coming from edges not in $A$ and removing any empty cycles created and $\phi'=(\sigma')^{-1} (\alpha')^{-1}$.     
\end{enumerate}
We write $\bG'=\bG[A]$ and say that $A$ \emph{induces} $\bG'$. 

Recall that an embedded graph is a quasi-tree if it has only one face, and, 
correspondingly, a  map  
$\bG=[\sigma,\alpha,\phi]$ is a \emph{quasi-tree} if $\phi$ is cyclic. (Notice that if $\phi$ is cyclic then $\bG$ is connected.) 
The \emph{spanning quasi-trees} of a map 
$[\sigma,\alpha,\phi]$
are its induced submaps $[\sigma',\alpha',\phi']$ which are quasi-trees with $\sigma'$ having the same number of cycles as $\sigma$.
This notion corresponds precisely with the analogous situation in embedded graphs. In particular, only connected maps have spanning quasi-trees. 

The following proposition records two standard and well-known properties of partial duality and spanning quasi-trees. (The first item appears to have been first  written down in~\cite{zbMATH05929399}, the second  in~\cite{zbMATH07094555}.)
Before stating the proposition, we recall that $X \bigtriangleup Y$ denotes the \emph{symmetric difference} $(X\cup Y) - (X\cap Y)$ of two sets $X$ and $Y$.

\begin{proposition}\label{prop:onevertex}
Let $\bG=[\sigma,\alpha,\phi]$ be a map and let $A$ and $T$ be two 
subsets of its edges. Then the following hold. 
\begin{enumerate}
\item $\bG[T]$ is a spanning quasi-tree of $\bG$ if and only if $\alpha^T\sigma$ is cyclic.
\item  $\bG[T]$ is a spanning quasi-tree of $\bG$ if and only if $\bG^A[T \bigtriangleup A]$ is a spanning quasi-tree of $\bG^A$.
\end{enumerate}
\end{proposition}
As this result is central to our work and the proofs in the literature are not in terms of maps, we sketch a proof for completeness.
\begin{proof}
If $T=\emptyset$, then the first part follows immediately from the definition, so we shall assume that $T\ne \emptyset$.
If $\bG[T]$ has fewer vertices than $\bG$, then it is not a spanning quasi-tree of $\bG$. Moreover $\alpha^T\sigma$ cannot be cyclic, as $\sigma$ has at least two cycles and there is a least one cycle of $\sigma$ not containing any of the darts associated with edges of $T$. So we may assume that $\bG[T]$ and $\bG$ have the same number of vertices. 
Suppose that $\bG[T]=[\sigma', \alpha',\phi']$.
Note that $\alpha'=\alpha^T$. 
The key observation is that $\alpha^T \sigma$ is cyclic if and only if $\alpha' \sigma'$ is cyclic. 
So $\phi' = (\alpha'\sigma')^{-1}$ is cyclic if and only if $\alpha^T\sigma$ is cyclic.
Thus the first part is true. The second part follows easily from the first.
\end{proof}

Notice that if $\bG$ is a disconnected map, then 
for every subset $T$ of the edges of $\bG$,
$\alpha^T \sigma$ is not cyclic. 

Finally, we note that the notation $\alpha^A$ is convenient for working with directed maps. 
Let $A$ be a subset of the edges of a directed map $G=[\sigma, \alpha, \phi]$, then the map obtained by 
reversing 
each edge in $A$ is $[\alpha^A \sigma \alpha^A, \alpha, \alpha^A\phi \alpha^A]$.

\begin{figure}[!t]
    \centering
        \subfloat[$\bG$.]{
             \labellist
\small\hair 2pt
\pinlabel {$a^-$} at   18 34
\pinlabel {$a^+$} at   94  44
\pinlabel {$b^-$} at     25.5 54
\pinlabel {$b^+$} at   36 27
\pinlabel {$c^-$} at  45 43
\pinlabel {$c^+$} at   70 33
\pinlabel {$d^-$} at  10 48
\pinlabel {$d^+$} at   92 25
\endlabellist
 \includegraphics[scale=1.2]{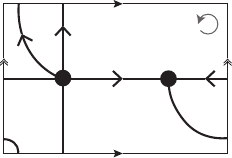} 
        \label{fex1}}
        \qquad
         \subfloat[$\bG$.]{
      \labellist
\small\hair 2pt
\pinlabel {$a$} at   67 13
\pinlabel {$b$} at      8 64 
\pinlabel {$c$} at   60 65
\pinlabel {$d$} at   10 24
\endlabellist
 \includegraphics[scale=1.2]{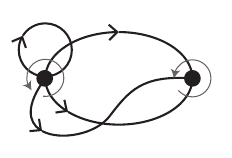} 
          \label{fex2} 
        } 
        
         \subfloat[$\bG^{\{c\}}$.]{
  \labellist
\small\hair 2pt
\pinlabel {$a$} at   9 15
\pinlabel {$b$} at    30 67
\pinlabel {$c$} at    67 47
\pinlabel {$d$} at    17 24
\endlabellist
 \includegraphics[scale=1.2]{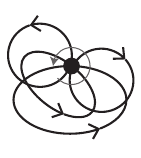} 
        \label{fex3}}
        \quad
         \subfloat[$\bG^{\{a,b,c\}}$.]{
  \labellist
\small\hair 2pt
\pinlabel {$a$} at    13 58
\pinlabel {$b$} at     52 18
\pinlabel {$c$} at     6 16 
\pinlabel {$d$} at     62 64
\endlabellist
 \includegraphics[scale=1.2]{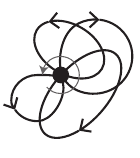} 
        \label{fex4}}
        \quad
 \subfloat[$\vec{M}(\bG)$]{
  \labellist
\small\hair 2pt
\pinlabel {$a$} at   46 69
\pinlabel {$b$} at     68 47 
\pinlabel {$c$} at    30 5  
\pinlabel {$d$} at    6 28
\endlabellist
 \includegraphics[scale=1.2]{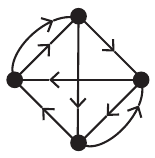} 
        \label{fex5}}
        
     \caption{Maps for Examples~\ref{basicexamp1},~\ref{basicexamp2},~\ref{ex:lap} and~\ref{ex:cf}.}
         \label{fex}
\end{figure}

\begin{example}\label{basicexamp1}
Let $\vec{\bG}=[\sigma,\alpha,\phi]$ be the directed map given by 
\begin{align*}\sigma&=(a^-b^+c^-b^-d^-)(a^+c^+d^+),\\ \alpha&=(a^-a^+)(b^-b^+)(c^-c^+)(d^-d^+), \\   \phi&=(a^- d^+ b^-)(a^+d^-c^+b^+c^-),\end{align*} and let $\bG$ be its underlying map. 
$\vec{\bG}$ is shown as a directed embedded graph (on the net of a torus) in Figure~\ref{fex1} and a directed rotation system in Figure~\ref{fex2}. 
Then $\bG$ has six spanning quasi-trees given by the edge sets
$\{a\},\{c\}, \{d\}, \{a,b,c\}, \{a,c,d\}, \{b,c,d\}$. 
The partial dual $\vec{\bG}^{\{c\}}$ is given by $[\alpha^{\{c\}}\sigma, \,\alpha,\, \sigma^{-1}\alpha^{\{a,b,d\}}],$ where 
\[\alpha^{\{c\}} \sigma= (a^- b^+ c^+ d^+ a^+ c^- b^- d^-).\]
Similarly, the partial dual $\vec{\bG}^{\{a,b,c\}}$ is given by $[\alpha^{\{a,b,c\}}\sigma, \,\alpha,\, \sigma^{-1}\alpha^{\{d\}}],$ where 
\[\alpha^{\{a,b,c\}} \sigma= (a^- b^- d^- a^+ c^- b^+ c^+ d^+).\]
The rotation systems for  $\vec{\bG}^{\{c\}}$ and $\vec{\bG}^{\{a,b,c\}}$ are shown in Figures~\ref{fex3} and~\ref{fex4}. 
Note that they are both bouquets since $\{c\}$ and $\{a,b,c\}$ are the edge sets of spanning quasi-trees of $\bG$. 
 (Figure~\ref{fex5} will be used in Example~\ref{ex:lap}.)
\end{example}

\subsection{Matrices}\label{danjkj}

Matrices here are over $\mathbb{R}$ with rows and columns indexed, in the same order, by some set $E$.    
For $X,Y\subseteq E$,  we let $A[X,Y]$ submatrix given by the rows indexed by $X$ and columns indexed by $Y$. 
We let $A[X]$ denote the \emph{principal submatrix} $A[X,X]$. The determinant of a principal submatrix is called a \emph{principal minor}. 
By convention $\det (A[\emptyset])=1$. 

A square matrix $A$ is \emph{skew-symmetric} if $A^t=-A$. (In this paper we always use a lower case $t$ in the exponent to denote the transpose of a matrix and  use an upper case $T$ in the exponent in relation to partial duality.)
It is \emph{unimodular} if it has determinant $1$ or $-1$, and  is \emph{principally unimodular (PU)} if  every nonsingular principal submatrix is unimodular.
It is a well-known result of Cayley that the determinant of an even order skew-symmetric matrix is the square of its Pfaffian. Thus an even order skew-symmetric and unimodular matrix must have determinant zero or one. Showing that an odd order skew-symmetric matrix must be singular is a standard exercise in matrix theory. We shall need these facts later.

We shall also make use of the Smith Normal Form of a matrix. We review a few relevant facts about it here and refer the reader to~\cite{MR3534076} for further background.
Let $A$ be an $m\times m$ integer matrix. Then there exist two $m\times m$ integer matrices $P$ and $Q$ that are invertible over $\mathbb{Z}$ and are such that $PAQ$ is a matrix $\diag(b_1, \ldots, b_r, 0, \ldots, 0)$,
  where $b_i$ is a strictly positive integer for all $i$ and $b_i$
divides $b_{i+1}$ for $1\leq i\leq r-1$. The matrix $B:=\diag(b_1, \ldots, b_r, 0, \ldots, 0)$ is the \emph{Smith Normal Form} of $A$. The integers $b_i$ are called the \emph{invariant factors} and are uniquely determined by $A$. Consequently pre- or post-multiplying $A$ by a unimodular integer matrix does not alter its Smith Normal Form. Notice that 
$\det (A)=\pm \det (B)$, so that, up to sign, $\det (A)$ is equal to the product of its invariant factors and when $A$ is nonsingular, the number of
invariant factors equals $m$.
The cokernel of $A$, $\mathbb{Z}^m/\langle A\rangle$, satisfies
\[
 \mathbb{Z}^m/\langle A\rangle  \cong \mathbb{Z}^{m-r}
 				\oplus \mathbb{Z}/b_1\mathbb{Z}
 				\oplus \cdots \oplus \mathbb{Z}/b_r\mathbb{Z}.
\]
It follows that if $A$ is invertible over $\reals$, then
\[ |\ints^m /\langle A\rangle| = |{\det (A)}|.\]
We now consider the case that the sum of the rows of $A$ and the sum of the columns of $A$ are both $0$, and $A$ has rank $m-1$ over the reals. 

Let $\hat A$ denote a matrix obtained from $A$ by deleting any row and any column. Then 
\begin{equation}\label{eq:lappedby}  |(\ints^{m}\cap \mathbf{1}^{\perp}) /\langle A \rangle|  = |\ints^{m-1}/\langle \hat A \rangle |  = |{\det (\hat A)}|,\end{equation}
where $\ints^{m}\cap \mathbf{1}^{\perp}$ is the subgroup of $\ints^m$ comprising $m$-tuples in which the entries sum to $0$.
Moreover,
\begin{equation}\label{eq:lapdance} \ints^m/\langle A \rangle  \cong \ints \oplus (\ints^{m}\cap \mathbf{1}^{\perp}) /\langle A \rangle \cong \ints \oplus \ints^{m-1}/\langle \hat A \rangle.\end{equation}
Hence, by Theorem~\ref{thm:classicalcritgroups}, we have
\[K(G)\cong  ( \ints^{|V|}\cap \mathbf{1}^{\perp} )/ \langle \Delta(G) \rangle  .\]

\section{The definition of the critical group of a map, and a summary of results}\label{sec:bullets}

In this section we introduce the critical group of a map.  We begin by associating a matrix with a map. This matrix plays a similar role to the signed cycle--cocycle matrix in the classical theory.

Let $\vec{\bG}=[\sigma,\alpha, \phi]$ be a directed, connected map described following Convention~\ref{con1}. Let $T$ be the edge set of a spanning quasi-tree of $\vec{\bG}$.

We are now ready to define a matrix associated with each map, from which we may easily compute the critical group of the map.
\begin{definition}\label{sadbh}
For a directed, connected map $\vec{\bG}=[\sigma,\alpha, \phi]$ with edge set $E$, together with the edge set $T$ of one of its spanning quasi-trees  we define $A(\vec{\bG},T)$ to be the $|E|\times |E|$-matrix whose rows and columns are indexed by the edges in $E$ and whose entries are either 0, 1 or $-1$ according to the following scheme. 
For edges $e$ and $f$, take the cyclic permutation $\alpha^{T}\sigma$ 
and remove all the entries except for the darts $e^-, e^+,f^-,f^+$. 
If the resulting permutation is $(e^+ f^+e^-f^-)$, then set $A_{e,f}=1$, if it is $(f^+e^+f^-e^-)$, then set $A_{e,f}=-1$,
otherwise set $A_{e,f}=0$.

In the case where $\vec{\bG}$ is a bouquet, $A(\vec{\bG})$ denotes the  matrix $A(\vec{\bG},\emptyset)$.
\end{definition}

 Notice that $A(\vec{\bG},T)$ is a skew-symmetric matrix. An example of a matrix $A(\vec{\bG},T)$ can be found in Example~\ref{basicexamp2}, with further examples given in Section~\ref{sec:examples}. 
The matrices $A(\vec{\bG},T)$ can be found in the literature, albeit in different settings. 
We summarise the appearances  in Remark~\ref{gdsag}.

The next two results record two useful facts for reference later. Each follows immediately from the definitions. Note that Proposition~\ref{prop:onevertex} implies that if $T$ is the edge set of a spanning quasi-tree of a map $\vec \bG$, then $\vec \bG^T$ is a bouquet.
\begin{proposition}\label{prop:submap}
Let $\vec \bB$ be a directed bouquet, and  $X$ be a subset of its edges.
Then $A(\vec \bB)[X] = A(\vec \bB[X])$.
\end{proposition}

\begin{lemma}\label{lem:dualtree}
Let $\vec \bG$ be a directed, connected map, and $T$ be the edge set of a spanning quasi-tree of $\vec \bG$. Then $A(\vec \bG^T)=A(\vec \bG,T)$. 
\end{lemma}

As previously mentioned, the matrices $A(\vec{\bG},T)$ play a  role similar to the signed cycle--cocycle matrix in the classical theory.
Although our primary focus is using the matrices $A(\vec{\bG},T)$ to define  critical groups, they have a rich structure, are of interest in their own right and appear implicitly in several places in the literature (see Remark~\ref{gdsag}). The following list summarises the main properties of the matrices that appear here. We provide references to the complete statements of the results.

\begin{itemize}
\item $A(\vec{\bG},T)$ is a principally unimodular matrix. (See Corollary~\ref{cor:matrixisPU}.)
\item  The rank of $A(\vec{\bB})$ equals  twice the genus of a directed bouquet $\vec{B}$. (See Theorem~\ref{thm:rank}.)
\item A directed bouquet $\vec{\bB}$ is a quasi-tree if and only if $A(\vec{\bB})$ is non-singular. (See Theorem~\ref{thm:Bouchetbouq}.)
\item When $\vec{\bG}$ is plane, $A(\vec{\bG},T)+I$ is the signed cycle--cocycle matrix. (See Theorem~\ref{prop:plane}.)
\item The characteristic polynomial of $A(\vec{\bB})$ enumerates the spanning quasi-trees of a bouquet $\vec{\bB}$ by genus. (See Theorem~\ref{charpolycount}.)
\item The Matrix--Quasi-tree Theorem: $\det(A(\vec{\bG},T)+I)$ equals the number of spanning quasi-trees in $\bG$. (See Theorem~\ref{basdh}.)
\item The Weighted Matrix--Quasi-tree Theorem: $\det(A(\vec \bB)Z+I_m)$ is the weighted quasi-tree generating polynomial of a bouquet $\bB$, where $Z=\diag(z_1,\ldots, z_m)$. 
(See Theorem~\ref{thm:wmqtt}.) This is extended to connected maps in general in Theorem~\ref{thm:wmqtt2}. 
Moreover, this polynomial is Hurwitz stable. (See Theorem~\ref{thm:hurst}.)
\item The order of the bicycle space of a connected map $\bG$ divides the number of spanning quasi-trees in $\bG$. (See Theorem~\ref{thm:bic}.)
\end{itemize}

\bigskip

Recall that the classical critical group of an abstract graph $G$ can be defined as the cokernel of the signed cycle--cocycle matrix. 
Our strategy for obtaining the critical group of a map is to construct a topological analogue of the signed cycle--cocycle matrix of an abstract graph and define the critical group as the cokernel of this matrix. A slight modification of the matrices $A(\vec{\bG},T)$ appearing above provides the required analogue.
For any directed connected map $\vec{\bG}$ and the edge set $T$ of any of its spanning quasi-trees  we can consider the cokernel of $A(\vec{\bG},T)+I$, where $I$ is the identity matrix of size $|E(\vec{\bG})|$. 
This defines a critical group for each pair  $\vec{\bG}$ and $T$. 
The matrix $A(\vec{\bG},T)+I$ does, in general, depend upon both the choice of directions of the edges, and the choice of spanning quasi-tree. 
Remarkably, however, the cokernel does not. 

\begin{theorem}\label{adhk}
Let $\vec{\bG}_1$ be a directed connected map and $\vec{\bG}_2$ be a directed connected map obtained from  $\vec{\bG}_1$ by redirecting  the edges in some (possibly empty) set. Further let $T_1$ be the edge set of a spanning quasi-tree of $\vec \bG_1$, and  $T_2$ be the edge set of a spanning quasi-tree of $\vec \bG_2$.
Then the cokernels of $A(\vec \bG_1,T_1)+I$ and $A(\vec\bG_2,T_2)+I$ are isomorphic.
\end{theorem}
Section~\ref{s:well} is devoted the the proof of this theorem.
With Theorem~\ref{adhk} ensuring independence of the choices made in its construction, we can now define the critical group of a map.
\begin{definition}\label{dcrit}
Let $\bG$ be a connected map with $m$ edges. Let $T$ be the edge set of any spanning quasi-tree of $\bG$, and let $\vec{\bG}$ be any directed map obtained by directing the edges of $\bG$. Then the \emph{critical group} of $\bG$, denoted by $K(\bG)$,  is defined as the cokernel of  the matrix $A(\vec{\bG},T)+I$, 
\[ K(\bG):= \mathbb{Z}^{m}/\langle A(\vec{\bG},T)+I\rangle,  \]
where $I$ is the $m\times m$ identity matrix.

If $\bG$ is a disconnected map, then its critical group is defined to be the direct sum of the critical groups of its connected components.
\end{definition}

We shall show that the following properties hold.
\begin{itemize}
\item The critical group is  well-defined. (See Theorem~\ref{adhk}.)
\item The order of the critical group of a connected map $\bG$ is equal to the the number of spanning quasi-trees in  $\bG$.  (See Theorem~\ref{order}.)
\item The critical groups of a map and its partial duals are all isomorphic. In particular, the critical group of a map and its dual are isomorphic. (See Theorem~\ref{dualiso}.)
\item When $\bG$ is a plane map, the critical group $K(\bG)$ of the map $\bG$ is isomorphic to the classical critical group $K(G)$ of its underlying abstract graph $G$. (See Theorem~\ref{prop:plane}.)
\item The critical group of a map can be understood as the classical critical group of its underlying abstract graph modified by a perturbation by the topological information from its embedding. (See Theorem~\ref{thm:pet}.)
\end{itemize}

Definition~\ref{dcrit} introduces the critical group of a map as the cokernel of a map version of the signed cycle--cocycle matrix. As noted in Theorem~\ref{thm:classicalcritgroups}, the critical group of an abstract graph $G$ can be defined as the cokernel of its reduced Laplacian. 
Analogously we can define the critical group of a map through the Laplacian of an associated abstract directed graph:
\begin{itemize}
\item The critical group of a map $\bG$ can be obtained as the the cokernel of the reduced Laplacian of its directed medial graph $\vec M(\bG)$:  \[K(\bG)\cong \ints^{m-1} /\langle \Delta^q(\vec M(\bG))\rangle
\cong 
(\ints^{m}\cap \mathbf{1}^{\perp})/\langle \Delta(\vec M(\bG))\rangle
.\] (See Theorem~\ref{thm:groupssame} and Section~\ref{s.Laplace} for definitions.)
\end{itemize}

We have so far provided two algebraic formulations of the critical group of a map. As we mentioned at the start of the introduction, in the classical case, the critical group of a graph can be constructed as the group of critical states of the chip-firing game (or sandpile model)~\cite{MR1676732,Chris+Gordon}.  
An immediate question is whether the critical group of a map can be realised in a similar way.
In Section~\ref{sec:chip} we show that indeed it can be.
\begin{itemize}
\item The critical group of a map $\bG$ can be realised as the group of critical states of a  chip-firing game on the edges of $\bG$ (See Theorem~\ref{psndf1}). 
\end{itemize}
In order to show this we describe a version of the chip-firing game in which chips are stacked on the edges of a map, rather than vertices as in the classical case. 
We then use work on the classical case, described for example in~\cite{MR2477390}, to identify the critical states of this chip-firing game with the elements of the critical group of a map.

\begin{example}\label{basicexamp2}
Let $\vec{\bG}=[\sigma,\alpha,\phi]$ be the directed map from Example~\ref{basicexamp1}, shown in Figure~\ref{fex1},
and let $\bG$ be its underlying map.
By direct computation,
\[  A(\vec{\bG},\{c\})= 
\begin{pmatrix} 0&-1&-1&-1\\1&0&0&1\\1&0&0&1\\1&-1&-1&0\end{pmatrix}
,\text{ and }
 A(\vec{\bG},\{a,b,c\})=
 \begin{pmatrix} 0&1&0&1\\-1&0&1&1\\0&-1&0&0\\-1&-1&0&0\end{pmatrix}
.\]
It can be checked that $A(\vec{\bG},\{c\})=A(\vec{\bG}^{\{c\}})$, that    $A(\vec{\bG},\{a,b,c\})=A(\vec{\bG}^{\{a,b,c\}})$, and also that, in the notation of Equations~\eqref{twist1} and~\eqref{twist2} explained later, $A(\vec{\bG}^{\{c\}})\ast \{a,b\} = A(\vec{\bG}^{\{a,b,c\}})$.

The Smith Normal Form of both $A(\vec{\bG},\{c\})+I_4$  and $A(\vec{\bG},\{a,b,c\})+I_4$ is $\diag(1,1,1,6)$ and so $K(\bG)\cong \mathbb{Z} /6\mathbb{Z}$. Observe that the order of $K(\bG)$ equals the number of spanning quasi-trees in $\bG$. 
\end{example}

\begin{remark}\label{gdsag}
As noted above, the matrices $A(\vec{\bG},T)$ can be found in the literature, albeit in different settings. 
 They appear in knot theory as the matrices $\mathrm{IM}(D)$ from Bar-Natan and Garoufalidis' celebrated proof of the Melvin--Morton--Rozansky Conjecture in~\cite{zbMATH00885522}. The connection is as follows. Given a chord diagram $D$ with chords numbered from 1 to as $m$ as described in \cite[Definition~3.4]{zbMATH00885522}. Construct a map  $\vec{\bG}$ by taking  $\alpha=(1^-1^+) \cdots (m^-m^+)$, and for $\sigma$, start with the empty cycle then read backwards along the skeleton of $D$ and the first time you meet a chord $i$ append $i^+$ to $\sigma$, and the second time $i^-$.  Then $\mathrm{IM}(D) = A(\vec{\bG},\emptyset)$. 
 This observation reappears in Theorem~\ref{thm:Bouchetbouq}.

They appear in topological graph theory as Mohar's oriented overlap matrices from~\cite{zbMATH04123755}. When  $\vec{\bB}$ is a directed bouquet the matrix $A(\vec{\bB},\emptyset)$ coincides with the oriented overlap matrices of $\vec{\bG}$ with respect to its unique spanning tree. (In Mohar's terminology the appearance of  $(e^+ f^+ e^- f^-)$ when restricting $\alpha^{T}\sigma$ to the darts $e^-, e^+,f^-,f^+$ means ``$e$ is inside $f$''.) Furthermore by making use of Proposition~\ref{prop:submap} and Lemma~\ref{lem:dualtree} it is not hard to see that when  $\vec{\bG}$ is a directed map with edge set $E$ and $T$ is the edge set of one of its spanning trees, then the submatrix $A(\vec{\bG},T)[E-T]$ coincides with the oriented overlap matrix of $\vec{\bG}$ formed with respect to the tree on $T$. 
Mohar's work appears later in Theorem~\ref{thm:rank}.

They appear in Macris and Pul\'e's work on Euler circuits in  directed graphs~\cite{MR1395471}. Suppose $\bB$ is an oriented bouquet on darts 
$1,\dots, 2n$ with $\sigma=(1\,2\,\ldots\,2n)$ and with $n$ edges.
Let $\vec{\bB}$ be the directed graph obtained from $\bB$ by taking the smaller dart in each to be the tail. 
Following Macris and Pul\'e's notation, if $V$ is the set of darts and $P$ is the set of transpositions then their matrix $I_P$ coincides with $A(\vec{\bB},\emptyset)$. 
It is readily seen that Macris and Pul\'e's directed graph $G_P$ is  the (underlying) directed medial graph of $\bB$. (And all their graphs $G_P$ arise in this way.)  The Euler circuits in this directed medial graph correspond to the spanning quasi-trees in $\bB$. These observations together with the main result of~\cite{MR1395471} offer an alternative approach to our Theorem~\ref{basdh}.

Most importantly for us, they appear in Bouchet's work on circle graphs~\cite{MR900943}. Suppose that $\vec{\bB}=[\sigma,\alpha, \phi]$ is a directed bouquet that follows  Convention~\ref{con1}. Let $\mu:=\sigma$ to match Bouchet's notation.
 The permutation $\mu$ is then what Bouchet terms a ``separation'' (of the double occurrence word obtained from $\mu$ by forgetting the exponents), and  Bouchet's matrix $A(\mu)$ is readily seen to equal $A(\vec{\bB},\emptyset)$. The observation that Bouchet's $\mu^*$ equals  $\phi^{-1}$ provides an interpretation the main result of~\cite{MR900943} in terms of dual maps and quasi-trees, an observation we make use of in Theorem~\ref{thm:Bouchetbouq} below.   
\end{remark}

\section{The group is well-defined }\label{s:well}
For much of this section, we will be considering a directed bouquet $\vec\bB$ and the matrix $A(\vec\bB,\emptyset)$. 
Recall that for brevity, we abbreviate $A(\vec\bB,\emptyset)$ to $A(\vec\bB)$.

The following result is a translation (see Remark~\ref{gdsag}) of Theorem~3.2 of~\cite{zbMATH04123755}. 
\begin{theorem}\label{thm:rank}
Let $\vec{\bB}$ be a directed bouquet. Then the genus of $\vec{\bB}$ equals $r(A(\vec{\bB}))/2$.
\end{theorem}

The next theorem is a rephrasing of the main theorem of~\cite{MR900943}. 
The first part can be deduced from Theorem~\ref{thm:rank} and was also proved (in several different ways) in~\cite[Theorem~3]{zbMATH00885522} via the connection given in Remark~\ref{gdsag}.

\begin{theorem}\label{thm:Bouchetbouq}
Let $\vec{\bB}$ be a directed bouquet. 
Then  $\vec{\bB}$ is a quasi-tree if and only if $A(\vec{\bB})$ is invertible. 
Moreover, when $A(\vec{\bB})$ is invertible, $A(\vec{\bB})^{-1}=A(\vec{\bB}^*)$. 
\end{theorem}

We record two corollaries of Theorem~\ref{thm:Bouchetbouq}.
\begin{corollary}\label{cor:matrixisPU}
Let $\vec{\bG}$ be a directed, connected map and $T$ be the edge set of a spanning quasi-tree of $\vec{\bG}$. Then $A(\vec{\bG},T)$ is principally unimodular (PU).
\end{corollary}
\begin{proof}
As $\vec{\bG}^T$ is a bouquet, Lemma~\ref{lem:dualtree} implies that it is sufficient to establish the result in the case where $\vec{\bG}$ is a bouquet and $T=\emptyset$.

To do this let $\vec{\bB}$ be a directed bouquet.
Suppose that $\det(A(\vec{\bB})[X])\ne 0$. 
Then, by Theorem~\ref{thm:Bouchetbouq} and Proposition~\ref{prop:submap}, 
$(A(\vec{\bB})[X])^{-1} = A((\vec{\bB}[X])^*)$. In particular, $(A(\vec{\bB})[X])^{-1}$ is integer-valued. Hence $\det(A(\vec{\bB})[X]) =\det(A(\vec{\bB}[X]))=\pm 1$. Thus $A(\vec{\bB})$ is PU.
\end{proof}

\begin{corollary}\label{cor:bouquets}
Let $\vec{\bB}$ be a directed bouquet and $X$ be a subset of the edges of $\bB$. Then 
\[ \det (A(\vec\bB)[X]) = \begin{cases} 1 & \text{if $X$ is the edge set of a spanning quasi-tree of $\vec\bB$,} \\ 0 & \text{otherwise.} \end{cases}\]
\end{corollary}

\begin{proof}
It follows from Theorem~\ref{thm:Bouchetbouq} and Proposition~\ref{prop:submap} that $\det (A(\vec\bB)[X])$ is non-zero if and only if $X$ is the edge set of a spanning quasi-tree of $\vec\bB$. It follows from Corollary~\ref{cor:matrixisPU} and the remarks in Subsection~\ref{danjkj} about the principle minors of skew-symmetric matrices, that if $\det (A(\vec\bB)[X])$ is non-zero, then it is equal to $1$.
\end{proof}

We say that edges $e$ and $f$ of a map are \emph{interlaced} in a cyclic permutation $\beta$ of its darts if removing all the entries of $\beta$ except for the darts 
$e^-, e^+,f^-,f^+$ results in the cyclic permutation $(e^+f^+e^-f^-)$ or $(e^+f^-e^-f^+)$. Furthermore, we say that edges $e$ and $f$ of a bouquet $[\sigma,\alpha,\phi]$ are \emph{interlaced} if they are interlaced in $\sigma$.
So edges $e$ and $f$ of a directed bouquet $\vec \bB$ are interlaced if and only if the $(e,f)$-entry of $A(\vec \bB)$ is non-zero.

The following lemmas are easy observations.

\begin{lemma}\label{lem:walknewtree}
Let $T$ be the edge set of a spanning quasi-tree of a directed connected map $\vec \bG=[\sigma,\alpha,\phi]$ and let $e$ and $f$ be edges of $\vec \bG$.
Suppose that
$\alpha^T \sigma$ has the form $(e^a W f^b X e^c Y f^d Z)$, where $\{a,c\}=\{b,d\}=\{+,-\}$, and $W$, $X$, $Y$ and $Z$ are (possibly empty) subwords of $\alpha^T \sigma$ collectively including all the darts arising from edges of $\bG$ other than $e$ and $f$. Then
\[ \alpha^{T\bigtriangleup \{e,f\}} \sigma 
 = (e^a W f^d Z e^c Y f^b X).\]
\end{lemma}

\begin{lemma}\label{lem:easyobsinterlace}
Let $T$ be the edge set of a spanning quasi-tree of a connected map $\bG=[\sigma,\alpha,\phi]$
and let $e$ and $f$ be edges of $\vec \bG$. 
Then $T':= T\bigtriangleup \{e,f\}$ is the edge set of a spanning quasi-tree of $\bG$ if and only if 
$e$ and $f$ are interlaced in $\alpha^T \sigma$.
\end{lemma}

\begin{proof}
 If $e$ and $f$ are interlaced in $\alpha^T \sigma$, then the result follows from the previous lemma. If $e$ and $f$ are not interlaced in $\alpha^T \sigma$, then it is easy to verify that $\alpha^{T\bigtriangleup \{e,f\}}\sigma$ has more than one cycle and therefore 
 $T\bigtriangleup \{e,f\}$ is not the edge set of a spanning quasi-tree of $\bG$.
\end{proof}

Suppose that $\bB=[\sigma,\alpha,\phi]$ is a bouquet and that $e$ and $f$ are interlaced. Our next goal is to describe how the matrices $A(\bB)$ and $A(\bB^{\{e,f\}})$ are related. 

For a matrix $A$ with rows and columns both indexed by a set $E$, let $A_{i\leftrightarrow j}$ denote the operation of interchanging the rows indexed by $i$ and $j$, and then interchanging the columns indexed by $i$ and $j$. Let $A_{-i}$ denote the operation of multiplying the row indexed by $i$ by $-1$ and then multiplying the column indexed by $i$ by $-1$. 

These matrix operations may be used to describe the effect on $A(\bB)$ of making various changes to $\bB$. 
Suppose $\bB'$ is formed from $\bB$ by interchanging the darts $i^{+}$ with $j^{+}$ 
and $i^{-}$ with $j^{-}$
in the map. Then  $A(\bB')=A(\bB)_{i\leftrightarrow j}$. 
If $\bB'$ is formed from $\bB$ by changing the direction of edge $i$, that is interchanging its head and tail, then
$A(\bB')=A(\bB)_{-i}$. Finally if $\bB'$ is formed from $\bB$ by reversing the orientation, that is replacing $\sigma$ by $\sigma^{-1}$, then $A(\bB')=-A(\bB)$.

Now, suppose that $A$ is a skew-symmetric $\{-1,0,1\}$-matrix with rows and columns indexed by $E$ and that $e,f \in E$. Furthermore suppose that 
\begin{equation}\label{twist1}
A= \begin{pNiceMatrix}[first-row,first-col] &e&f&E' \\ e& 0 & z & u \\ f&-z & 0 & v \\ E'&-u^t & -v^t & A'
\end{pNiceMatrix},
\end{equation}
where $E'=E-\{e,f\}$, $z \in \{-1,1\}$, $u$ and $v$ are row vectors indexed by $E'$ and the rows and columns of $A'$ are indexed by $E'$. 
Let 
\begin{equation}\label{twist2}
A*\{e,f\} = \begin{pNiceMatrix}[first-row,first-col] 
&e&f&E'\\
e&0 & -z & -zv \\ 
f&z & 0 & zu \\ E' & z v^t & - zu^t & A''\end{pNiceMatrix},
\end{equation}
where 
$A'' = A'+z(v^t u - u^t v)$. 
Notice that $A*\{e,f\}$ is skew-symmetric.  We say that $A*\{e,f\}$ is the \emph{pivot} of $A$ about $e$ and $f$. This is a special case of the principal pivot transform~\cite{MR3396730,MR0114760}.

\begin{lemma}\label{lem:pivots}
Let $\vec \bB$ be a directed bouquet with interlaced edges $e$ and $f$. Then $A(\vec \bB^{\{e,f\}}) = A(\vec \bB)*\{e,f\}$. 
\end{lemma}

\begin{proof}
Suppose that $\vec \bB=[\sigma,\alpha,\phi]$ and that \[\sigma = (e^a W f^b X e^c Y f^d Z),\] where $\{a,c\}=\{b,d\}=\{+,-\}$, and $W$, $X$, $Y$ and $Z$ are (possibly empty) subwords of $\sigma$ including all the darts arising from edges of $\bB$ other than $e$ and $f$. 
Now suppose that $\vec \bB*\{e,f\}=[\sigma',\alpha,\phi']$. Then from Lemma~\ref{lem:walknewtree}, we have
\[ \sigma' 
 = (e^a W f^d Z e^c Y f^b X).\]

The matrix operations discussed just before the lemma commute with pivoting in the following sense.
\begin{enumerate}
    \item For all $i$, $(A_{-i})*\{e,f\} = (A*\{e,f\})_{-i}$.
    \item For all distinct $i$ and $j$ with neither equal to $e$ or $f$, $(A_{i\leftrightarrow j})*\{e,f\} = (A*\{e,f\})_{i\leftrightarrow j}$.
    \item $(A_{e\leftrightarrow f})*\{e,f\} = (A*\{e,f\})_{e\leftrightarrow f}$.
    \item $(-A)*\{e,f\}=-(A*\{e,f\})$.
\end{enumerate}
These observations enable us to cut down the number of cases that we need to check. As both $A(\bB^{\{e,f\}})$ and $(A(\bB)*\{e,f\})$ are skew-symmetric, we focus only on the entries above the diagonal.

We first show that the entries in the 
rows and columns of $A(\vec \bB^{\{e,f\}})$ indexed by $e$ and $f$ are as claimed. Certainly 
$A(\bB^{\{e,f\}})_{e,f}=-A(\bB)_{e,f}$.

Now let $g$ be an edge of $\vec \bB$ other than $e$ or $f$. To verify that 
$A(\bB^{\{e,f\}})_{e,g}=(A(\bB)*\{e,f\})_{e,g}$ and
$A(\bB^{\{e,f\}})_{f,g}=(A(\bB)*\{e,f\})_{f,g}$,
it suffices to consider the cyclic permutations $\hat \sigma$ and $\hat \sigma'$ obtained from $\sigma$ and $\sigma'$, respectively, by removing all the darts other than $e^+$, $e^-$, $f^+$, 
$f^-$, $g^+$ and $g^-$. 
Because of the operations 1.--4. above that commute with pivoting we may assume that $\hat\sigma$ is one of $(e^+g^+g^-f^+e^-f^-)$, $(e^+g^+f^+g^-e^-f^-)$, and
$(e^+g^+f^+e^-g^-f^-)$. 
Then $\hat\sigma'$ is respectively $(e^+g^+g^-f^-e^-f^+)$, $(e^+g^+f^-e^-f^+g^-)$ and  $(e^+g^+f^-e^-g^-f^+)$, and one may check that 
$A(\bB^{\{e,f\}})_{e,g}=(A(\bB)*\{e,f\})_{e,g}$ and
$A(\bB^{\{e,f\}})_{f,g}=(A(\bB)*\{e,f\})_{f,g}$ in each case.

Now let $g$ and $h$ be distinct edges of $\vec 
\bG$ other than $e$ or $f$. To verify that $A(\bB^{\{e,f\}})_{g,h}=(A(\bB)*\{e,f\})_{g,h}$ it suffices to consider the cyclic permutations $\hat \sigma$ and $\hat \sigma'$ obtained from $\sigma$ and $\sigma'$, respectively, by removing all the darts other than $e^+$, $e^-$, $f^+$, 
$f^-$, $g^+$, $g^-$, $h^+$ and $h^-$. It is not difficult to check that if 
at least one of the edges in one of the sets $\{e,f\}$ and $\{g,h\}$ is not interlaced with at least one of the edges in the other set, then
\[ A(\bB^{\{e,f\}})_{g,h}=A(\bB)_{g,h}=(A(\bB)*\{e,f\})_{g,h}.\]
Because of this observation and the operations that commute with pivoting we may assume that $\hat\sigma$ is one of 
\[\begin{array}{ccc}
(e^+g^+f^+h^+e^-g^-f^-h^-),& 
(e^+g^+h^+f^+e^-g^-h^-f^-),& 
(e^+g^+h^+f^+e^-h^-g^-f^-),\\
(e^+g^+h^+f^+g^-e^-h^-f^-),& 
(e^+g^+h^+f^+h^-e^-g^-f^-),& 
(e^+g^+h^+f^+g^-e^-f^-h^-),\\ 
(e^+g^+h^+f^+h^-e^-f^-g^-). &&
\end{array}\]
In each case it may be checked that
$A(\bB^{\{e,f\}})_{g,h}=(A(\bB)*\{e,f\})_{g,h}$, as required. 
\end{proof}

The next lemma is the key step in showing that the critical group is well-defined.
\begin{lemma}\label{lem:smith}
Let $\vec \bB$ be a directed bouquet with interlaced edges $e$ and $f$. Then $A(\vec \bB)+I$ and $A(\vec \bB^{\{e,f\}})+I$ have the same invariant factors.
\end{lemma}

\begin{proof}
By Lemma~\ref{lem:pivots} 
it is sufficient to show that $A(\vec \bB)+I$ and $A(\vec \bB)*\{e,f\}+I$ have the same invariant factors. Suppose that
\[ A(\vec \bB)+I= \begin{pNiceMatrix}[first-row,first-col] &e&f&E' \\ e& 1 & z & u \\ f&-z & 1 & v \\ E'&-u^t & -v^t & A'
\end{pNiceMatrix}.\]
If necessary, by multiplying the row indexed by $f$ by $-1$ and then the column indexed by $f$ by $-1$, operations which do not change the invariant factors, we may assume that $z=1$. 

Interchanging the rows indexed by $e$ and $f$, and then interchanging the columns indexed by $e$ and $f$ gives
\[\begin{pNiceMatrix}[first-row,first-col] &e&f&E' \\ e& 1 & -1 & v \\ f&1 & 1 & u \\ E'&-v^t & -u^t & A'
\end{pNiceMatrix}.\]
Now for each $h$ in $E'$, subtract $u_h$ times the sum of the columns indexed by $e$ and $f$ from the column indexed by $h$ to give
\[\begin{pNiceMatrix}[first-row,first-col] &e&f&E' \\ e& 1 & -1 & v \\ f&1 & 1 & -u \\ E'&-v^t & -u^t & B
\end{pNiceMatrix},\]
where $B_{g,h}=A'_{g,h}+(u_g+v_g)u_h$.  Finally for each $g$ in $E'$, add $u_g$ times the difference between the row indexed by $f$ and the row indexed by $e$  to the row indexed by $g$ to give
\[\begin{pNiceMatrix}[first-row,first-col] &e&f&E' \\ e& 1 & -1 & v \\ f&1 & 1 & -u \\ E'&-v^t & u^t & B'
\end{pNiceMatrix},\]
where 
\[B'_{g,h}=B_{g,h}+u_g(-u_h-v_h)= A'_{g,h}+(u_g+v_g)u_h + u_g(-u_h-v_h).\]
Thus $B'=A'+z(v^tu-u^tv)$, and the matrix we have obtained is $A(\vec \bB)*\{e,f\}+I$. Thus $A(\vec \bB)+I$ and 
$A(\vec \bB)*\{e,f\}+I$ have the same invariant factors.
\end{proof}

Before finally proving Theorem~\ref{adhk}, we recall that if $T_1$ and $T_2$ are the edge sets of spanning quasi-trees of a map, then $|T_1\bigtriangleup T_2|$ is even. This can be seen by noting that Proposition~\ref{prop:onevertex} implies that both $\alpha^{T_1}\sigma$ and $\alpha^{T_2}\sigma$ are cyclic. As 
$\alpha^{T_2}\sigma =
\alpha^{T_1\bigtriangleup T_2}\alpha^{T_1}\sigma$, 
the permutation 
$\alpha^{T_1\bigtriangleup T_2}$ must be even.

\begin{proof}[Proof of Theorem~\ref{adhk}]
First $A(\vec \bG_2,T_2)+I$ may be obtained from 
$A(\vec \bG_1,T_2)+I$ by multiplying some of the rows and columns by $-1$, an operation which does not change the invariant factors. Thus the cokernels of $A(\vec \bG_1,T_2)+I$ and $A(\vec \bG_2,T_2)+I$ are isomorphic.

We now prove by induction on $|T_1\bigtriangleup T_2|$ that 
$A(\vec \bG_1,T_1)+I$ and $A(\vec \bG_1,T_2)+I$ have the same invariant factors. 
Let $\vec \bB_1$ and $\vec \bB_2$ be the directed bouquets $\vec \bG_1^{T_1}$ and $\vec \bG_1^{T_2}$, and note that  $\vec \bB_2 = \vec \bB_1^{T_1 \bigtriangleup T_2}$.
If $T_1=T_2$, then there is nothing to prove, so we may suppose that $T_1 \ne T_2$.

By the first part of Proposition~\ref{prop:onevertex},
$|T_1\bigtriangleup T_2|$ is the edge set of a spanning quasi-tree of $\vec \bB_1$, so by Theorem~\ref{thm:Bouchetbouq}, the matrix $A(\vec \bB_1)[T_1\bigtriangleup T_2]$ is invertible. Hence it must contain at least one non-zero entry implying that there are edges $e$ and $f$ in $T_1\bigtriangleup T_2$ which are interlaced in $\vec\bB_1$. 
By Lemma~\ref{lem:easyobsinterlace},
$\{e,f\}$ is the edge set of a spanning quasi-tree of $\vec\bB_1$.
Let $T_3:= T_1 \bigtriangleup \{e,f\}$
and $\vec\bB_3:=\vec\bB_1^{\{e,f\}}=\vec\bG_1^{T_3}$. By Lemma~\ref{lem:smith}, 
$A(\vec \bB_1)+I$ and
$A(\vec \bB_3)+I$ have the same invariant factors, so  Lemma~\ref{lem:dualtree} implies that $A(\vec \bG_1,T_1)+I$ and
$A(\vec \bG_1,T_3)+I$
have the same invariant factors. Now $|T_3 \bigtriangleup T_2| < |T_1 \bigtriangleup T_2|$, so we may apply the inductive hypothesis to deduce that 
$A(\vec \bG_1,T_3)+I$ and
$A(\vec \bG_1,T_2)+I$
have the same invariant factors. Hence
$A(\vec \bG_1,T_1)+I$ and
$A(\vec \bG_1,T_2)+I$
have the same invariant factors.
\end{proof}

Our next result generalizes work of Cori and Rossin~\cite[Theorem~2]{MR1756151} who established the case where $\bG$ is a plane map and also of Berman~\cite[Proposition~4.1]{MR819700} who essentially proved the plane case but without mentioning critical groups.
\begin{theorem}\label{dualiso}
Let $\bG$ be a connected map and $A$ be a subset of its edges. Then the critical groups of $\bG$ and $\bG^A$ are isomorphic. In particular, the critical groups of $\bG$ and $\bG^*$ are isomorphic.
\end{theorem}
\begin{proof}
Let $T$ be the edge set of a spanning quasi-tree of $\bG$. Then by Proposition~\ref{prop:onevertex}, $T\bigtriangleup A$ is the edge set of a spanning quasi-tree of $\bG^A$. Moreover, $(\bG^A)^{T\bigtriangleup A}$ and $\bG^T$ are the same bouquet. Thus, by Lemma~\ref{lem:dualtree}, $A(\vec \bG,T)+I = A(\vec \bG^A,T\bigtriangleup A)+I$ and the result follows.
\end{proof}

In our final result of this section we show that the critical group is independent of the surface orientation. More precisely suppose that $\embG_1$ is a connected embedded graph and $\embG_2$ is obtained from $\embG_1$ by reversing the surface embedding. Let $\bG_1=[\sigma_1,\alpha_1,\phi_1]$ and $\bG_2=[\sigma_2,\alpha_2,\phi_2]$ be the connected maps obtained from $\embG_1$ and $\embG_2$ respectively. Then
$\sigma_2=\sigma_1^{-1}$ and $\alpha_2=\alpha_1$. So $\phi_2=\alpha_1\phi_1^{-1}\alpha_1$. We show that the critical groups of $\bG_1$ and $\bG_2$ are isomorphic.
Before proving the theorem, we note that if $\bB_1=[\beta,\alpha,\beta^{-1}\alpha]$ is a bouquet, then 
$\bB_2=[\alpha^T\beta^{-1}\alpha^T,\alpha,\alpha^T\beta\alpha^{E-T}]$
is a bouquet. Moreover, if the edges of $\bB_1$ and $\bB_2$ are directed 
to give $\vec{\bB}_1$ and $\vec{\bB}_2$, respectively, 
so that the same darts are heads of edges in both directed bouquets, then  $A(\vec{\bB}_2)$ is obtained from $A(\vec{\bB}_1)^t$ 
by multiplying the elements in rows indexed by $T$ by $-1$ and then multiplying the elements in columns indexed by $T$ by $-1$.

\begin{theorem}\label{thm:planarsame}
Let $\bG_1=[\sigma,\alpha,\sigma^{-1}\alpha]$ and $\bG_2=[\sigma^{-1},\alpha,\sigma\alpha]$ be connected maps. Then the critical groups of $\bG_1$ and $\bG_2$ are isomorphic.
\end{theorem}
\begin{proof}
Direct the edges of $\bG_1$ and $\bG_2$ so that the head of each edge is the same dart in each map. Observe that $T$ is the edge set of a spanning quasi-tree of $\bG_1$ if and only if it is the edge set of a spanning quasi-tree of $\bG_2$. Furthermore, $\alpha^T\sigma^{-1}=\alpha^T(\alpha^T\sigma)^{-1}\alpha^T$.
So, by the remarks before the theorem, $A(\bG_2,T)$ may be obtained from $A(\bG_1,T)^t$, and similarly $A(\bG_2,T)+I$ may be obtained from $A(\bG_1,T)^t +I$, by multiplying the elements in rows indexed by $T$ by $-1$ and then multiplying the elements in columns indexed by $T$ by $-1$. 
The matrices $A(\bG_1,T)+I$ and $A(\bG_1,T)^t +I$ have the same Smith Normal form, as one may obtain the Smith Normal form of $A(\bG_2,T)^t +I$ by applying the same sequence of operations as one would apply to obtain the Smith Normal form of  $A(\bG_1,T)+I$, but replacing each column operation by the corresponding row operation and vice versa. 
Finally we note that $A(\bG_2,T)+I$ has the same Smith Normal form as 
$A(\bG_1,T)^t +I$. So the critical groups of $\bG_1$ and $\bG_2$ are isomorphic, as required.
\end{proof}

\section{Connections with the classical critical group}\label{sec:connections}
In this section we compare the critical group of a map with the classical critical group of an abstract graph. We begin by showing that if $\bG$ is a plane map, then its critical group is isomorphic to the classical critical group of its underlying abstract graph.

\begin{lemma}\label{lem:remarksbefore}
Let $\vec \bG$ be a directed connected map and let $T$ be the edge set of a spanning tree of the   underlying abstract directed graph $\vec G$ of $\vec \bG$. Let $e$ be in $T$. Then the row of $A(\vec {\mathbb G},T)+I$ indexed by $e$ is the signed incidence vector of the fundamental cocycle $C^*(T,e)$ of $\vec G$.
\end{lemma}

\begin{proof}
Suppose that $\bG=[\sigma,\alpha,\phi]$ is the underlying map of $\vec \bG$.  
Then $\alpha^{\{e\}}\alpha^T  \sigma$ has two cycles, one containing $e^+$ and the other containing $e^-$. Let $C^+$ be the cycle containing $e^+$ and $C^-$ be the cycle containing $e^-$. Let $A:=A(\mathbb G,T)$.
For $f\ne e$, it is easy to check that
\[ A_{e,f} = \begin{cases} 1 & \text{if $f^+ \in C^+$ and $f^- \in C^-$,}\\
-1 & \text{if $f^+ \in C^-$ and $f^- \in C^+$,}\\
0 & \text{otherwise.}\end{cases}\]
By Lemma~\ref{lem:easyobsinterlace}, if $A_{e,f}\ne 0$, then $T\bigtriangleup \{e,f\}$ is the edge set of a spanning quasi-tree of $\bG$. But no spanning quasi-tree has fewer edges than a spanning tree, so 
$f \notin T$. Therefore no edge of $T$ other than $e$ has one dart in $C^+$ and the other in $C^-$.

We can write $\sigma = \alpha^{T-\{e\}}(\alpha^{e}\alpha^T\sigma)$. For each edge in $T-\{e\}$, its two constituent darts lie in the same cycle of $\alpha^{e}\alpha^T\sigma$, so we can partition the cycles of $\sigma$ into two sets $S^+$ and $S^-$, so that for $i\in \{+,-\}$ the darts belonging to one of the cycles in $S^i$ are precisely those belonging to $C^i$. The edges with one dart in $S^+$ and one dart in $S^-$ are precisely the edges of the fundamental cocircuit $C^* (T,e)$ of $\vec G$. For each such edge $f$, its sign is positive 
in the signed incidence vector 
if and only if $f^+$ is in $S^+$.

  Thus we see that the row of $A(\vec {\mathbb G},T)+I$ indexed by $e$ is the signed incidence vector of the fundamental cocycle $C^*(T,e)$.
\end{proof}

We will require the following well-known lemma on abstract directed graphs.
\begin{lemma}\label{lem:cycoplanar}
    Let $G$ be a directed connected abstract graph and let $T$ be the edge set of a spanning tree of $G$, containing edge $f$ but not edge $e$. 
Then $f \in C(T,e)$ if and only if $e \in C^*(T,f)$. Moreover when these both occur, $e$ and $f$ have the same sign in the signed incidence vector of $C(T,e)$ if and only if they have opposite signs in the signed incidence vector of $C^*(T,f)$.
\end{lemma}

It is now easy to verify the following theorem.
\begin{theorem}\label{prop:plane}
Let $\vec \bG$ be a directed, connected plane map and $T$ be the edge set of a spanning tree of $\vec \bG$. Then the signed cycle--cocycle  matrix of 
the underlying abstract directed graph of $\vec \bG$
and $T$ is $A(\vec \bG,T)+I$. 

Therefore the critical group of a plane map is isomorphic to the classical critical group of its underlying abstract graph.
\end{theorem}
\begin{proof}
Let $M:=M(\vec {\mathbb G},T)$ and $A:=A(\vec{\mathbb G},T)$.
By Lemma~\ref{lem:remarksbefore}, for each edge $e$ of $T$, the rows corresponding to $e$ in $M$ and $A+I$ are equal.

Now suppose that $e\notin T$. 
As $M$ depends only on the underlying abstract directed graph of $\vec{\mathbb G}$,
Lemma~\ref{lem:cycoplanar} implies that
for each $f\ne e$
\[ M_{e,f} = \begin{cases} 0 & \text{if $f\notin T$,} \\ -M_{f,e} & \text{if $f\in T$.}\end{cases}\]
As $\vec{\mathbb G}$ is plane, the edge sets of the spanning quasi-trees of $\vec{\mathbb G}$ are precisely the edge sets of the spanning trees of the underlying abstract directed graph $\vec G$. So if $f \notin T$, then $T\bigtriangleup\{e,f\}$ is not the edge set of a spanning quasi-tree of $\vec{\mathbb G}$. Thus Lemma~\ref{lem:easyobsinterlace} implies that $A_{e,f}=0$. If $f\in T$, then as $A$ is skew-symmetric, $A_{e,f}=-A_{f,e}$. The first part of the theorem follows and the second is an immediate consequence.
\end{proof}

In Corollary~\ref{planeconv} we shall see that the critical group of a map coincides with the classical critical group of its underlying abstract graph if and only if the map is plane.

\begin{remark}
In~\cite[Proposition 5.7]{zbMATH01116184} Bouchet proved a result concerning multimatroids which may be translated to delta-matroids and then applied to maps as follows. Let $\bG$ be a connected map with edge set $E$ and let $T$ be the edge set of a spanning quasi-tree of $\bG$. For every edge $e$ of $E$ there is a unique minimal pair $(X,Y)$ of disjoint subsets of $E- \{e\}$ so that 
\begin{enumerate}
    \item if $e\in T$, then there is no set $T'$ with  
$X \subseteq T' \subseteq (E- \{e\})-Y$, so that $T'$ is the edge set of a spanning quasi-tree of $\bG$;

    \item if $e\notin T$, then there is no set $T'$ with  
$X\cup\{e\} \subseteq T' \subseteq E-Y$, so that $T'$ is the edge set of a spanning quasi-tree of $\bG$.
\end{enumerate}
Furthermore $T\bigtriangleup\{e,f\}$ is the edge set of a spanning quasi-tree of $\bG$ if and only if $f \in X\cup Y$. The set $X\cup Y\cup\{e\}$ is called the \emph{fundamental circuit of $T$ and $e$}. 

Combining these observations with Lemma~\ref{lem:easyobsinterlace}, we see that each row of 
$A(\vec \bG,T)+I$ is an appropriately signed incidence vector of the fundamental circuit of $T$ and the edge indexing the row. So the critical group is constructed using the row space of a matrix formed by appropriately signed fundamental circuits, in much the same way as it can be for abstract graphs. 
\end{remark}

\medskip
 
Let $G$ be a connected abstract graph with edge set $E$ of size $m$ and $T$ be the edge set of a spanning tree of $G$ with $r$ edges. Then, by Lemma~\ref{lem:cycoplanar},
 the abstract directed graph $\vec G$ has signed circuit--cocycle matrix  of the form
   \[
M(\vec G,T)=   \begin{pNiceMatrix}[first-row,first-col] &T&E- T\\
T&I_r& B \\
E-T&-B^t & I_{m-r}\\
\end{pNiceMatrix}.
\]

\begin{proposition}\label{prop:criel}
Let  $G=(V,E)$ be a connected abstract graph with $m$ edges and $T$ be the edge set of a spanning tree of $G$ with $r$ edges. Let $\vec G$ be any directed graph obtained from $G$ by giving a direction to each of its edges and let 
$B= M(\vec G,T)[T,E- T]$. Then, for the critical group $K(G)$ we have the following isomorphisms.
\[      K(G) \cong \ints^{m-r} /\langle B^t B+I_{m-r}\rangle  \cong \ints^{r} /\langle B B^t+I_{r}\rangle.\]
  \end{proposition}
 
  \begin{proof} 
  Recall that $K(G):=  \ints^{m} / \langle M(\vec G,T) \rangle$.
 We have  
   \[
\begin{pmatrix} 
I_r& 0 \\
B^t & I_{m-r}\\
\end{pmatrix}
\begin{pmatrix} 
I_r& B \\
-B^t & I_{m-r}\\
\end{pmatrix}
\begin{pmatrix} 
I_r& -B \\
0 & I_{m-r}\\
\end{pmatrix}
=
\begin{pmatrix} 
I_r& 0 \\
0 & B^t B+I_{m-r}\\
\end{pmatrix}.
\]
As the integer matrices  
$\begin{pmatrix} 
I_r& 0 \\
B^t & I_{m-r}\\
\end{pmatrix}$ and
$\begin{pmatrix} 
I_r& -B \\
0 & I_{m-r}\\
\end{pmatrix}$
are unimodular, the matrices $M(\vec G,T)$ and
$B^t B+I_{m-r}$
have the same non-trivial invariant factors. Thus 
$K(G) \cong \ints^{m-r} /\langle B^t B+I_{m-r}\rangle$. A similar argument shows that $K(G) \cong \ints^{r} /\langle B B^t+I_{r}\rangle$.
 \end{proof}

 \medskip
  
 If an abstract loopless graph $G$ has an apex vertex $q$, that is a vertex sending a single edge to every other vertex, then we may take $T$ to be the spanning tree whose edges are the edges of $G$ incident with $q$. Then, 
 the matrix $B B^t+I_r$ is just  the reduced Laplacian of the graph $\Delta_q(G)$ 
 and  we get the usual presentation of the critical group as 
 $K(G)\cong \ints^r/\langle \Delta_q(G)\rangle$, see~\cite{MR1676732}. We extend this approach to maps to get another presentation of the critical group of a map.

Lemma~\ref{lem:cycoplanar} implies the following.
\begin{lemma}
 Let $\vec \bG$ be a directed, connected map with edge set $E$ and $T$ be the edge set of a spanning tree of $\vec \bG$ with $r$ edges. Let $G$ be the underlying abstract directed  graph of $\vec \bG$ and let $B=M(\vec G,T)[T, E- T]$. Then 
 \[
    A(\vec \bG,T)+I= \begin{pNiceMatrix}[first-row,first-col] &T&E- T\\
T&I_r& B \\
E- T&-B^t & D\\
\end{pNiceMatrix}.
\]
Moreover, the matrix $(I_r\,|\, B)$ is the totally unimodular representation of the cycle matroid of the underlying abstract graph of $\bG$ and $D':=D-I$, is PU.
 \end{lemma}
Thus, the information needed to recover $A(\vec \bG,T)+I$ is the matrices $B$ and $D'$.  
 
\begin{theorem}\label{thm:pet}
Let $\vec \bG$ be a directed, connected map with edge set $E$ comprising $m$ edges, and let $T$ be the edge set of a spanning tree of $\vec \bG$ with $r$ edges. 
Let $B=A(\vec \bG,T)[T, E- T]$, $D'=A(\vec \bG,T)[E-T]$ and $D=D'+I_{m-r}$. Then 
\[ K(\bG) \cong  \ints^{m-r}/\langle B^t B+D\rangle. \]
\end{theorem}
\begin{proof}
Postmultiplying  
$A(\vec \bG,T)+I$ by the unimodular matrix $\begin{pmatrix}I_r& -B \\
0 & I_{m-r}
\end{pmatrix}
$ yields
  \[
\begin{pmatrix} 
I_r& B \\
-B^t & D
\end{pmatrix}
\begin{pmatrix} 
I_r& -B \\
0 & I_{m-r}
\end{pmatrix}
=
\begin{pmatrix} 
I_r& 0 \\
-B^t & D+B^t B
\end{pmatrix}.
\]
Thus the matrices $A(\vec \bG,T)+I$ and $D+B^t B$ have the same invariant factors, except possibly for the addition or removal of factors equal to one.
\end{proof}
 
Taken together, the previous theorem and Proposition~\ref{prop:criel} explain formally our intuition that the critical group of a map is the classical critical group of the underlying abstract graph modified by a perturbation due to the topological information describing the embedding.

\section{Enumerating quasi-trees and the order of the critical group}
\label{sec:qtrees}

The spanning trees of an abstract graph  play an important role in the theory of the classical critical group. In particular, the order of the critical group of an abstract graph is equal to its number of spanning trees. 
Key to this fact are the connections between the reduced Laplacian and spanning trees, as demonstrated by the Matrix--Tree Theorem (stated here as Theorem~\ref{kmt}). 

Recent work in the areas of graph polynomials~\cite{zbMATH06824436,zbMATH05960754,zbMATH07395405, zbMATH05876837} and matroids~\cite{zbMATH07094555,zbMATH07055741} hints at a general principle that spanning quasi-trees play the role of spanning trees when changing setting from abstract graph theory to topological graph theory. In this section we see further evidence towards such a general principle by demonstrating intimate relationships between spanning quasi-trees, the matrices $A(\vec{\bG},T)$ and the critical groups for maps that run parallel to the classical relationships between spanning trees, the reduced Laplacian $\Delta^q(G)$, and the critical group of an abstract graph.   

\medskip

We begin with the surprising result that that the characteristic polynomial of $A(\vec{\bB})$ enumerates the spanning quasi-trees of a bouquet $\bB$  with respect to genus. We use $P_A(t)$ to denote the \emph{characteristic polynomial}, $\det(t I_m-A)$, of an $m\times m$ matrix $A$.

  \begin{theorem}\label{charpolycount}
  Let $\bB$ be a bouquet with edge set $E$, and suppose that $\vec{\bB}$ is obtained by arbitrarily directing the edges of $\bB$.  Then 
     \[
 t^{|E|}  P_{A(\vec{\bB})}(t^{-1}) =   \sum_{\bT} t^{2g(\bT)},
 \]
   where the sum is over all spanning quasi-trees of $\bB$, and where $g(\bT)$ denotes the genus of $\bT$.
 \end{theorem}
\begin{proof}
For any $m\times m$ matrix $A$ it is well-known (see, for example,~\cite[Section~1.2]{MR2978290}) that 
\begin{equation}\label{pshv}
   P_A(t)= \sum_{k=0}^m (-1)^{m-k}E_{m-k}(A)t^k,
 \end{equation}
where  $E_k(A)$   denotes the sum of the principal minors of $A$ of size $k$. 
Recall from Subsection~\ref{danjkj} that a skew-symmetric matrix of odd order is necessarily singular. Combining this observation with Corollary~\ref{cor:bouquets}, we see that
\begin{equation}\label{pshv2} t^mP_{A(\vec \bB)}(t^{-1})= \sum_{k=0}^m t^k E_k(A(\vec \bB)) 
= \sum_{\substack{X \subseteq E \\ \vec\bB [X] \text{ is a quasi-tree}}} t^{|X|}.  
    \end{equation}
When $\vec \bB[X]$ is a quasi-tree
 it has one vertex, $|X|$ edges, and one face and therefore has genus $(2-1+|X|-1)/2 = |X|/2$.
\end{proof}

\begin{example}
Let $\bG$ and $\vec{G}$ be as in Example~\ref{basicexamp1}. Then $\bG^{\{c\}}$ and $\bG^{\{a,b,c\}}$ are both bouquets. 
The matrix $A(\vec{\bG}^{\{c\}})$ has characteristic polynomial $t^4+5t^2$, and so by Theorem~\ref{charpolycount} the generating function for the spanning quasi-trees of $\bB$  with respect to genus is $5t^2+1$. It can be checked directly that the edge set of the spanning quasi-tree of genus zero is $\emptyset$, and the edge sets of those of genus one are $\{a,b\}$, $\{a,c\}$, $\{a,d\}$, $\{b,d\}$ and $\{c,d\}$.

Similarly, $A(\vec{\bG}^{\{a,b,c\}})$ has characteristic polynomial $t^4+4t^2+1$, so by Theorem~\ref{charpolycount} the generating function for the spanning quasi-trees of $\bB$  with respect to genus is $t^4+4t^2+1$. Again it can be checked directly that  edge set of the spanning quasi-tree of genus zero is $\emptyset$, and the edge sets of those of genus one are $\{a,b\},\{a,d\},\{b,c\},\{b,d\}$, and the one of genus two is $\{a,b,c,d\}$.
\end{example}

Theorem~\ref{charpolycount} shows that  the characteristic polynomial of $A(\vec{\bB})$ enumerates the spanning quasi-trees of a bouquet $\bB$ by genus. 
As every connected map is a partial dual of a bouquet, 
one may hope to extend Theorem~\ref{charpolycount} to an arbitrary map $\bG$ by forming a partial dual $\bG^T$ that is a bouquet and considering the characteristic polynomial of $A(\vec{\bG}^T)=A(\vec{\bG},T)$. But, as partial duals may have different numbers of spanning quasi-trees of any given genus, $t^{|E|}  P_{A(\vec{\bG},T))}(t^{-1})$ does \emph{not} in general enumerate the spanning quasi-trees of $\bG$ by genus. 
However, since partial duality does not change the total number of spanning trees, evaluating this polynomial at $t=1$ gives the total number of spanning quasi-trees in $\bG$. 

 This observation leads to an analogue of the Matrix--Tree Theorem for maps, expressing the number of  spanning quasi-trees in a map as the determinant of a matrix.   
 The resulting ``Matrix--Quasi-tree Theorem'', stated below as Theorem~\ref{basdh} is implicit in the literature. As outlined in Remark~\ref{gdsag}, by considering the directed medial graph, it can be deduced from Macris and Pule's work~\cite[Theorem~1]{MR1395471} on counting Eulerian circuits (see also~\cite{MR1428586,MR1762101}).

\begin{theorem}[Matrix--Quasi-tree Theorem]\label{basdh}
Let $\bG$ be a connected map on $m$ edges, $T$ be the edge set of any spanning quasi-tree of $\bG$, and $\vec{\bG}$ be obtained by directing the edges of $\bG$. 
Then  $\bG$ contains exactly 
$\det(A(\vec{\bG},T)+I_m)$ spanning quasi-trees.
\end{theorem}
\begin{proof}
As $T$ is the edge set of a spanning quasi-tree, by Proposition~\ref{prop:onevertex} the partial dual $\vec{\bG}^T$ is a bouquet. 
Applying Theorem~\ref{charpolycount} to this bouquet, then taking $t=1$,  gives that 
  $\det (I_m - A(\vec{\bG}^T))$ equals the number of spanning quasi-trees in $\bG^T$. 
Since $A(\vec{\bG}^T)$ is skew-symmetric and the determinant is invariant under transposition we may write this as 
$\det ( A(\vec{\bG}^T)+I_m)$.
 
Next by Lemma~\ref{lem:dualtree}, $A(\vec{\bG}^T) = A(\vec{\bG},T)$, so $\det(A(\vec{\bG},T)+I_m)$ equals the number of spanning quasi-trees of $\bG^T$. 
The result then follows upon noting that by Proposition~\ref{prop:onevertex} the maps $\bG$ and $\bG^T$ contain the same number of spanning quasi-trees.
\end{proof}

It is worth highlighting that, following the notation in the statement of Theorem~\ref{basdh}, when $\bG$ is plane,  $\det(A(\vec{\bG},T)+I_m)$ counts the number of  spanning trees of $\bG$.

\medskip

Through an application of Theorem~\ref{basdh} we can  determine the order of the critical group of a map.

\begin{theorem}\label{order}
The order of the critical group $K(\bG)$ of a connected map $\bG$ is equal to the number of spanning quasi-trees in $\bG$.
\end{theorem}
\begin{proof}
$K(\bG)$  is defined as $\mathbb{Z}^{m}/\langle A(\vec{\bG},T)+I_m\rangle$, where  $\vec{\bG}$ is obtained by directing the $m$ edges of $\bG$ and $T$ is the edge set of any spanning quasi-tree.
By Theorem~\ref{basdh} $\det (A(\vec{\bG},T)+I_m)$ equals the number of spanning quasi-trees in $\bG$. In particular, it is non-zero.
As $A(\vec{\bG},T)+I_m$ is non-singular, it follows from the remarks in Section~\ref{danjkj} that
$|K(\bG)|=|\mathbb{Z}^{m}/\langle A(\vec{\bG},T)+I_m\rangle|=\det (A(\vec{\bG},T)+I_m)$.
\end{proof}

Combining Theorems~\ref{prop:plane} and~\ref{order}  gives the following result.
\begin{corollary}\label{planeconv}
The critical group $K(\bG)$ of a connected map $\bG$ is isomorphic to the critical group $K(G)$ of its underlying abstract graph $G$ if and only if  $\bG$ is plane.
\end{corollary}
\begin{proof}
Every non-plane connected map contains a non-plane spanning quasi-tree. 
(To see this, let $T$ be the edge set of a spanning tree of $\bG^*$ and let $E-T$ be its complementary set. By Proposition~\ref{prop:onevertex},
$\bG[E-T]$ is a spanning quasi-tree of $\bG$. Using Equation~\eqref{eqn:genus} and the definition of $\bG^*$, $\bG[E-T]$ has strictly more edges than a spanning tree of $\bG$ unless $\bG$ has genus zero.)
Furthermore, every spanning tree of $\bG$ is a spanning quasi-tree.
Thus by Theorems~\ref{classicalorder} and~\ref{order}, when $\bG$ is non-plane $|K(G)|<|K(\bG)|$. The converse is Theorem~\ref{prop:plane}. 
\end{proof}

\begin{remark}\label{savgh}
Much in the proofs of Theorems~\ref{charpolycount} and~\ref{basdh} depend on properties of skew-symmetric PU-matrices rather than properties of maps. 
Extracting the matrix parts of the arguments in fact proves a more general result about delta-matroids. We keep our remarks concise and refer to~\cite{zbMATH07307106} for an introduction to delta-matroids.

Delta-matroids were introduced by Bouchet in~\cite{MR904585} and generalise matroids by effectively allowing bases of different sizes. 
A delta-matroid consists of a pair $(E,\mathcal{F})$, where $E$ is a finite set and $\mathcal{F}$ is a non-empty collection of its subsets, satisfying a particular exchange property whose definition we do not include here. The elements of $\mathcal{F}$ are called \emph{feasible sets}. A matroid is exactly a delta-matroid in which all feasible sets have the same size. A delta-matroid is \emph{normal} if the empty set is feasible. Delta-matroids also have a notion of a partial dual (often called a twist) which is an analogue of the map partial dual (see~\cite{zbMATH07094555}). 

If $A$ is a skew-symmetric matrix over any field $F$, then a delta-matroid $D(A)=(E,\mathcal{F})$ can be obtained by taking $E$ to be a set of row labels of the matrix, and taking $X\in \mathcal{F}$ if and only if $A[X]$ is non-singular. If a delta-matroid has a partial dual isomorphic to $D(A)$ then it is \emph{representable} over $F$. It is \emph{regular} if it is representable over every field. Geelen~\cite{MR2694410} showed  that a delta-matroid is regular if and only if we can take the matrix $A$ to be a skew-symmetric PU-matrix over the reals. The proofs of Theorems~\ref{charpolycount} and~\ref{basdh} thus give results about regular delta-matroids.

Following the proof of Theorem~\ref{charpolycount} until Equation~\eqref{pshv2} shows that the characteristic polynomial enumerates the feasible sets in a normal regular delta-matroid by size. Specifically, if $A$ is an $m\times m$ skew-symmetric PU matrix over the reals, then 
\[t^m  P_A(t^{-1}) = \sum_{F\in \mathcal{F}}  t^{|F|}.\]
Since partial duality does not change the number of feasible sets of delta-matroids, by taking $t=1$ in this equation gives an analogue of Theorem~\ref{basdh}: a regular delta-matroid has exactly $\det (A+I)$ feasible sets, where $A$ is a representing real skew-symmetric PU-matrix for one of its partial duals. 
\end{remark}

\begin{remark}\label{dfth}
In the light of Remark~\ref{savgh} it is natural to ask if the critical group of a map $\bG$ depends only on the delta-matroid $D(A(\vec\bG,T))$. Rather surprisingly the answer is no.

Let $\bG_1=[\sigma_1, \alpha, \phi_1 ]$ for  
$\sigma_1=( 1^+ 2^+ 3^+ 4^+ 5^+ 6^+ 1^-3^-2^-4^-6^-5^-)$
and $\bG_2=[\sigma_2, \alpha, \phi_2 ]$ for 
$\sigma_2 =(  1^+ 2^+ 3^+ 6^+5^+4^+ 1^- 3^- 2^- 5^- 6^- 4^-)$, 
where $\alpha$ is given by Convention~\ref{con1} and the $\phi_i$ are implicit.
Then it can be checked that $K(\bG_1)= \mathbb{Z}/3\mathbb{Z}  \oplus  \mathbb{Z}/6\mathbb{Z}$ and $K(\bG_2)=\mathbb{Z}/18\mathbb{Z}$.  
The delta-matroids $D(A(\bG_1))$ and $D(A(\bG_2))$ are equal, because the matrices $A(\bG_1)$ and $A(\bG_2)$ are both PU and are equal modulo 2. Indeed the maps $\bG_1$ and $\bG_2$ are 2-isomorphic in the sense of~\cite{MR4224067}  and hence give rise to the same delta-matroid. 

However, Geelen showed~\cite[Theorem~6.2]{MR2694410} that every normal  3-connected regular delta-matroid
is represented by a unique PU-matrix up to multiplying rows and columns by $-1$. It follows from this and the fact that the principal pivot transform~\cite{MR3396730, MR2694410} on matrices preserves the PU-property,
that we can define the critical group for 3-connected regular delta-matroids. We omit the details. 
\end{remark}

\bigskip

We now established weighted versions of Theorem~\ref{basdh}. Given a connected map $\bG$ with edge set $E=\{1,\ldots,m\}$ we associate an indeterminate $z_i$ with each edge $i$. With each subset $X$ of $E$, we associate a monomial $Z(X):=\prod_{i\in X}z_i$. Then the \emph{weighted quasi-tree generating polynomial} is defined by 
\[ Q(\bG):= \sum_{\substack{X\subseteq E: \\ \bB[X] \text{ is a quasi-tree}}}   Z(X).
\]
Suppose that we arbitrarily direct the edges of $\bB$ to obtain $\vec \bB$. Define the matrix $Z:=\diag(z_1,\ldots,z_m)$. 
\begin{theorem} [Weighted Matrix--Quasi-tree Theorem for bouquets]\label{thm:wmqtt}
Let $\bB$ be a bouquet with $m$ edges and let $\vec \bB$ be obtained from $\bB$ by arbitrarily directing its edges. Then
\[ Q(\bB) = \det(A(\vec \bB)Z+I_m).\]
\end{theorem}

\begin{proof}
We have
\[ \det(A(\vec \bB)Z+I_m) = \sum_{X\subseteq E} \det ((A(\vec \bB)Z)[X]),\]
where $(A(\vec \bB)Z)[X]$ denotes the principal submatrix $A(\vec \bB)Z$ 
containing the rows and columns corresponding to edges in $X$.
Furthermore for any subset $X$ of $E$, we have
\[ \det (A(\vec \bB)Z)[X] = \det (A(\vec \bB)[X]) Z(X).\]
The result now follows from Corollary~\ref{cor:bouquets}.
\end{proof}

The general case follows easily.
\begin{theorem}[Weighted Matrix--Quasi-tree Theorem]\label{thm:wmqtt2}
Let $\bG$ be a connected map with $m$ edges, let $\vec \bG$ be obtained from $\bG$ by arbitrarily directing its edges and let $T$ be the edge set of a spanning quasi-tree of $\bG$. Then
\[  Q(\bG) = \det(A(\vec \bG,T)Z_T+I_m)Z(T),\]
where $Z_T$ is the $m\times m$ diagonal matrix so that
\[ z_{i,i} = \begin{cases} z_i & \text{if $i \notin T$,}\\
z_i^{-1} & \text{if $i \in T$.}\end{cases}\]
\end{theorem}

As $Q(\bG)$ is a multivariate polynomial in the  variables $z_1$, \ldots, $z_m$, we may consider questions concerning the location of its complex roots.  
Because $A(\vec \bB,T)$ is skew-symmetric, we may deduce the following result from the remarks after Theorem~4.2 of~\cite{MR2353258}. Recall that a multivariate polynomial with complex coefficients is \emph{Hurwitz stable} if it is non-zero when each of its variables has strictly positive real part.

\begin{theorem}\label{thm:hurst}
Let $\bG$ be a connected map. Then $Q(\bG)$ is Hurwitz stable.
\end{theorem}

\bigskip

We close this section by showing that a result of Berman~\cite{MR819700} concerning abstract graphs also holds for maps. Given an abstract graph $G$, its (binary) cycle space $\mathcal C_2(G)$ is the 
subspace of $\ints_2^{|E(G)|}$ 
generated by the incidence vectors of its cycles. Its orthogonal complement $\mathcal C_2^{\perp}(G)$
is the (binary) cocycle space of $G$, that is, the subspace of $\ints_2^{|E(G)|}$ generated by the incidence vectors of its cocycles. 
Its \emph{bicycle space} is $\mathcal B_2(G):=\mathcal C_2(G)\cap \mathcal C_2^{\perp}(G)$.
Berman~\cite{MR819700} showed that if $G$ is connected, then $|\mathcal B_2(G)|$ divides the number of spanning trees of $G$.

As $\mathcal C_2(G)$ and $\mathcal C_2^{\perp}(G)$ are orthogonal complements, we have 
$\mathcal B_2(G)= (\mathcal C_2(G) + \mathcal C_2^{\perp}(G))^{\perp}$.
It is well-known~\cite{Chris+Gordon} that if $T$ is the edge set of a spanning tree of $G$, then 
the rows of $M_2(G,T)$, the binary matrix obtained from $M(\vec G,T)$ by ignoring the signs, form a basis for $\mathcal C_2(G) + \mathcal C_2^{\perp}(G)$.
Thus $\mathcal B_2(G)$ is the kernel of $M_2(G,T)$.
It does not seem to be possible to define analogues of the cycle and cocycle space for maps, but the preceding discussion suggests a way to define the bicycle space of a map.  

For a connected map $\bG$ and edge set $T$ of a spanning quasi-tree of $\bG$, we define the matrix $A_2(\bG,T)$ to be the binary matrix obtained from $A(\vec \bG,T)$ by ignoring the signs. (This does not depend on the direction of the edges in $\vec \bG$, so it is a function of the underlying map $\bG$.) 
It may be shown directly using elementary row operations
that if $M$ is a symmetric binary matrix having zero diagonal
with $M_{e,f}\ne 0$, then the binary rowspaces of $I+M$ and $I+(M*\{e,f\})$ coincide. Consequently $\ker(I+M)=\ker(I+(M*\{e,f\}))$. Now, it is straightforward to apply the methods of Section~\ref{s:well}, to show that $\ker (A_2(\bG,T)+I)$ does not depend on the choice of $T$. 
So we define the bicycle space $\mathcal B_2(\bG):= 
\ker (A_2(\bG,T)+I)$, where $T$ is the edge set of an arbitrary spanning quasi-tree of $\bG$. The proof of the next result follows Berman's proof~\cite{MR819700} of the analogous result for abstract graphs.

\begin{theorem}\label{thm:bic}
Let $\bG$ be a connected map. Then $|\mathcal B_2(\bG)|$ divides the number of spanning quasi-trees of $\bG$.
\end{theorem}

\begin{proof}
By definition, $|\mathcal B_2(\bG)|=|\ker (A_2(\bG,T)+I)|$, where $T$ is the edge set of an arbitrary spanning quasi-tree of $\bG$. Let $l$ denote the (binary) nullity of 
$A_2(\bG,T)+I$. Then $l$ equals the number of even invariant factors of $A(\vec \bG,T)+I$, where $\vec \bG$ is obtained from $\bG$ by arbitrarily directing its edges. Theorem~\ref{order} implies that $2^l$ divides the number of spanning quasi-trees of $\bG$ and the result follows.
\end{proof}

\section{The critical group through the Laplacian}\label{s.Laplace}
Recall from Section~\ref{sec:prelim} that the classical critical group of an abstract graph can be defined as the cokernel of its signed cycle--cocycle matrix  or of its reduced Laplacian. 
Definition~\ref{dcrit} introduces the critical group of a map as the cokernel of a map version of the signed cycle--cocycle matrix. We now give an analogue of the Laplacian definition and show that these two definitions yield isomorphic groups.

Let $\bG=[\sigma,\alpha,\phi]$ be a connected map with set $D$ of darts. For a dart $d$, let $e(d)$ denote the edge $(d,\alpha(d))$.
The \emph{directed medial graph} $\vec M(\bG)$ is an abstract directed graph, with a vertex corresponding to each edge of $\bG$ and for each dart $d$ a directed edge from $e(d)$ to $e(\sigma(d))$. Notice that $\vec M(\bG)$ may contain parallel edges and loops, is strongly connected and every vertex of $\vec M(\bG)$ has both indegree and outdegree equal to $2$. Furthermore, for any set $A$ of edges of $\bG$, we have $\vec M(\bG^A)=\vec M(\bG)$. (Note that, in the literature, in contrast with our usage, medial graphs are typically embedded graphs.) 

Given an abstract directed graph $\vec G$ with $n$ vertices, its adjacency matrix $\Adj(\vec G)$ has rows and columns indexed by the vertices of $\vec G$, and $(i,j)$-entry equal to the number of edges of $\vec G$ with tail $i$ and head $j$. Let $d_i$ denote the outdegree of vertex $i$. Then the Laplacian $\Delta(\vec G)$ is given by $\Delta(\vec G):= \diag(d_1,\ldots,d_n)-\Adj(\vec G)$. Thus every row and column of $\Delta(\vec M(\bG))$ has sum zero, $\Delta(\vec M(\bG))$ has rank $n-1$ over the reals, and if $\vec M(\bG)$ is loopless, then each diagonal element of $\Delta(\vec M(\bG))$ is equal to $2$.  

\begin{definition}\label{dahdh}
Given a connected map $\bG$ with $m$ edges, we let $K_{\Delta}(\bG)$ denote the group defined by \[K_{\Delta}(\bG):= (\ints^m\cap \mathbf{1}^{\perp})/\langle \Delta(\vec M(\bG))\rangle.\]  
\end{definition}
It follows from Equation~\eqref{eq:lapdance} that
\begin{equation} \label{eq:lapboring}
K_{\Delta}(\bG)\cong \ints^{m-1}/\langle \Delta^q(\vec M(\bG))\rangle.
\end{equation}
Thus Equation~\eqref{eq:lappedby} gives
\begin{equation} \label{eq:lapsize}
|K_{\Delta}(\bG)| =  \det(\Delta^q(\vec M(\bG))).
\end{equation}

\begin{example}\label{ex:lap}
Consider the map $\bG$  in Figure~\ref{fex1}, and its 
 directed medial graph $\vec M(\bG)$ which is  shown in Figure~\ref{fex5}.
Its Laplacian is 
\[\begin{pmatrix*} 2 &-1&-1&0\\ 
0 & 2 & -1 & -1\\
0 & -1 & 2 & -1\\
-2 & 0 & 0 & 2
\end{pmatrix*}.
\] It is easily checked that its Smith Normal Form is $\diag(1,1,6,0)$, and that $ K_{\Delta}(\bG)=\mathbb{Z}/6\mathbb{Z}$.
\end{example}

The main result of this section is the following.
\begin{theorem}\label{thm:groupssame}
Let $\bG$ be a connected map. Then $K(\bG)\cong K_{\Delta}(\bG)$.
\end{theorem}

Before proving the theorem we show that the groups $K(\bG)$ and $K_{\Delta}(\bG)$ have the same size.

\begin{lemma}\label{lem:criticalssamesize}
Let $\bG$ be a connected map. Then $|K(\bG)|=|K_{\Delta}(\bG)|$.
\end{lemma}

\begin{proof} 
The proof follows by combining some classical results with Theorem~\ref{basdh}. First note that $|K(\bG)|$ is equal to the number of spanning quasi-trees of $\bG$, by Theorem~\ref{basdh}.
Next the number of spanning quasi-trees of $\bG$ equals the number of Euler tours in $\vec M(\bG)$ by  Bouchet~\cite[Corollary~3.4]{MR1020647}.  
As every vertex of $\vec M(\bG)$ has both indegree and outdegree $2$, 
by the special case of the BEST Theorem~\cite{MR47311} proved by Tutte and Smith in~\cite{MR1525117}, 
the number of Euler tours of $\vec M(\bG)$ equals the number of arborescences rooted at an arbitrary vertex $q$ of $\vec M(\bG)$.
Then, by the Matrix--Tree Theorem for directed graphs~\cite{MR1579119,MR27521}, the number of arborescences rooted at $q$ is equal to the determinant of the reduced Laplacian, $\det(\Delta^q(\vec M(\bG)))$, which by Equation~\eqref{eq:lapsize} is $|K_{\Delta}(\bG)|$.
\end{proof}

\begin{proof}[Proof of Theorem~\ref{thm:groupssame}]
It follows from Theorem~\ref{dualiso} and our earlier observation that $\vec M(\bG^A)=\vec M(\bG)$ for any set $A$ of edges of $\bG$, that we need only consider the case where $\bG$ is a bouquet. 
Let $\vec \bB$ $=[\sigma,\alpha, \phi]$ be a directed bouquet with set $D$ of darts, and let $\bB$ be its underlying map. Suppose that $E(\vec \bB)=\{1,\ldots,m\}$. For $i=1,\ldots,m$, let $u_i$ denote the element of $\ints^m$ with $1$ in position $i$ and zeros elsewhere. 

Our aim is to define a group homomorphism  $f:\ints^m \rightarrow \ints^m\cap \mathbf{1}^{\perp}$ which will induce an isomorphism from $K(\bB)$ to $K_{\Delta}(\bB)$. 
To define $f$, we first introduce a map $g$ on the set of darts of $\bB$ and determine some of its key properties.
Let $g:D\rightarrow \ints^m\cap \mathbf{1}^{\perp}$ so that $g(d)=u_{e(\sigma(d))}-u_{e(d)}$. (Recall that $e(d)$ is the edge containing the dart $d$.)

Suppose that $d_1\ldots d_k$ is a subword of $\sigma$. Then
\[ \sum_{i=1}^{k-1} g(d_i) = u_{e(d_k))} - u_{e(d_1)}.\] 
In particular, if $d_k=\alpha(d_1)$, then $e(d_k)=e(d_1)$ and $\sum_{i=1}^{k-1} g(d_i)=0$.
Furthermore, for each edge $i$ 
\[ g(i^+) + g(i^-) = u_{e(\sigma(i^+))} + u_{e(\sigma(i^-))}
-2u(i),\]
which is equal to $i$-th row of $\Delta(\vec M(\bB))$ multiplied by $-1$. So $g(i^+)+g(i^-)\in \langle \Delta(\vec M(\bB)) \rangle$. 

Now define $f:\ints^m\rightarrow \ints^m\cap \mathbf{1}^{\perp}$ by setting $f(u_i)=g(i^+)$ and extending $f$ linearly to all elements of $\ints^m$. Clearly $f$ is a homomorphism. 
We will show that $f$ induces an isomorphism 
\[\hat f:\ints^m / \langle A(\vec{\mathbb B})+I\rangle \rightarrow  
(\ints^m\cap \mathbf{1}^{\perp})/\langle \Delta(\vec M(\bB))\rangle,\]
so that 
\[\hat f(x+\langle A(\vec{\mathbb B})+I\rangle) = f(x) + \langle \Delta(\vec M(\bB))\rangle.\]
Before we proceed, we introduce some notation.
Given a set $D'$ of darts, the edges of $\bB$ to which elements of $D'$ belong can be partitioned into three sets: 
\begin{align*}
    E_+&= \{e\in E(\vec \bB): e^+\in D',\ e^-\notin D'\}.\\
    E_-&= \{e\in E(\vec \bB): e^-\in D',\ e^+\notin D'\}.\\
    E_{+-} &= \{e\in E(\vec \bB): e^+\in D',\ e^-\in D'\}.
\end{align*}
(To avoid clutter we omit the dependency on $D'$ in this notation. This will always be clear from context.)

We first show that $\hat f$ is well defined.  As $f$ is a homomorphism, to do this it is sufficient to show that for all $i=1,\ldots,m$, if $x$ is the $i$-th row of 
$A(\vec{\mathbb B})+I$, 
then $f(x) \in \langle \Delta(\vec M(\bB))\rangle $.

Consider the subword $d_1=i^+d_2\ldots d_k=i^-$ of $\sigma$. Let $D':=\{d_1,\ldots,d_{k-1}\}$ and define $E_+$, $E_-$ and $E_{+-}$ as above.
Note that
\[ f(x) = \sum_{e\in E_+}  f(u_e)  - \sum_{e\in E_-} f(u_e).\]
Hence
\begin{align*}
 f(x) =& f(x) - \sum_{i=1}^{k-1} g(d_i)\\ 
=& \sum_{e\in E_+}  f(u_e)  - \sum_{e\in E_-} f(u_e) \\
& {} - \Big(\sum_{e\in E_+} g(e^+) + \sum_{e\in E_-} g(e^+) 
+ \sum_{e\in E_{+-}} (g(e^+)+ g(e^-))\Big)
\\=& \sum_{e\in E_+} (f(u_e) - g(e^+)) 
- \sum_{e\in E_-} (f(u_e) + g(e^-)) \\
& {} - \sum_{e\in E_{+-}} (g(e^+) + g(e^-)).\end{align*}
We have $f(u_e)=g(e^+)$ so each term in the first sum is zero and each term in the second and third sums has the form $g(e^+) + g(e^-)$ which belongs to $\langle \Delta(\vec M(\bB)) \rangle$.
So $f(x) \in \langle \Delta(\vec M(\bB))\rangle$ and $\hat f$ is well-defined. 

As $f$ is a homomorphism, it is clear that $\hat f$ is a homomorphism.

We now show that $\hat f$ is onto. It is sufficient to show that for any distinct edges $i$ and $j$ of $\bB$, we have $u_j-u_i + \langle \Delta(\vec M(\bB))\rangle \in \im(\hat f)$. Consider the subword $d_1=i^+d_2\ldots d_k=j^+$ of $\sigma$. Let $D':=\{d_1,\ldots,d_{k-1}\}$ and define $E_+$, $E_-$ and $E_{+-}$ as above.

Then
\begin{align*}  \MoveEqLeft{u_j-u_i - \Big(\sum_{e\in E_+} f(u_e) - \sum_{e\in E_-} f(u_e)\Big)}\\
&= \sum_{i=1}^{k-1} g(d_i) - \Big(\sum_{e\in E_+} f(u_e) - \sum_{e\in E_-} f(u_e)\Big)\\
&= \sum_{e\in E_+} (g(e^+)-f(u_e))
+ \sum_{e\in E_-} (g(e^-)+f(u_e))
+ \sum_{e\in E_{+-}} (g(e^+)+g(e^-)).\end{align*}
Each term in the first sum is zero and each term in the second and third sums has the form $g(e^+) + g(e^-)$, which belongs to $\langle \Delta(\vec M(\bB)) \rangle$. 
Thus \[u_j-u_i - \Big(\sum_{e\in E_+} f(u_e) - \sum_{e\in E_-} f(u_e)\Big) \in \langle \Delta(\vec M(\bB))\rangle.\]
It follows that 
\[ \hat f \Big(\sum_{e\in E_+} u_e - \sum_{e\in E_-} u_e\Big)
=u_j-u_i + \langle \Delta(\vec M(\bB))\rangle,\]
as required.

As $\hat f$ is onto, it follows from  Lemma~\ref{lem:criticalssamesize} that it is a bijection. Thus $K(\bB)$ and $K_{\Delta}(\bB)$ are isomorphic, as claimed. 
\end{proof}

\section{Chip firing}
 \label{sec:chip}
 
In this section we show that the expression for the critical group in terms of the Laplacian of the directed medial graph allows us to develop a chip-firing game for maps. We begin by briefly recalling the definition of the chip-firing game in both abstract graphs, where it has received most attention, and directed graphs, and its link with the critical group.

Within the literature~\cite{MR1676732,MR1120415,zbMATH06804236,MR3793659,MR1044086,MR2477390,zbMATH06952170,zbMATH01256098}, there are many variants on the chip-firing game, each with sometimes subtle differences in the definition, but informally speaking they are all based on the same idea: a number of chips is placed on each vertex of a graph and at any stage of the game, a vertex with at least as many chips as it has incident edges can fire by sending one chip along each of its incident edges to its neighbouring vertices. The game may stop should  no vertex be able to fire or may be infinite.

From the many variants of chip-firing in undirected graphs we describe the \emph{dollar game}. Our exposition follows~\cite{Chris+Gordon}. For simplicity we assume that our graphs are loopless.
Given a loopless, connected graph $G$ with vertex set $V$ and a distinguished vertex $q$, a \emph{state} of $G$ is a function $s:V\rightarrow \ints$ such that $s(v)\geq 0$ for all $v\in V-\{q\}$ and $\sum_{v\in V} s(v)=0$. 
A vertex other than $q$ can \emph{fire} (explained below) if $s(v) \geq d(v)$, where $d(v)$ is the degree of $v$. 
The vertex $q$ can fire if and only if no other vertex can fire.
When a vertex $v$ fires, the state $s$ is transformed to a new state $s'$, so that 
\[ s'(u):=   \begin{cases} s(u)-d(u) & \text{if $u=v$,}\\
s(u) + \Adj(G)_{v,u} & \text{otherwise,}
\end{cases}\]
where $\Adj(G)$ is the adjacency matrix of $G$.

A state $s$ is \emph{critical} if $s(v) < \deg(v)$ for all $v\in V-\{q\}$ and there is a non-empty sequence of firings transforming $s$ back to itself.
A slight adaptation of Theorem~14.11.1 of~\cite{Chris+Gordon} shows that if there is such a sequence, then there is one in which every vertex fires once. By regarding critical states as vectors in $\ints^n$ indexed by the vertices of $G$, the following is then shown in~\cite{Chris+Gordon}.

\begin{theorem}
\label{ajkw}
Let $G$ be a loopless, connected graph with $n$ vertices. Then every coset of $\langle \Delta(G)\rangle$ in $\ints^n \cap \mathbf{1}^{\perp}$ contains precisely one critical state. It follows that we can define a group on the set of critical states of $G$, by defining the sum of two critical states $s$ and $s'$ to be the unique critical state 
in the coset of $\langle \Delta(G)\rangle$ in $\ints^n \cap \mathbf{1}^{\perp}$
containing $s+s'$ and that this group is isomorphic to $K(G)$.
\end{theorem}

We now move on to chip-firing in Eulerian directed graphs, following the exposition in~\cite{zbMATH06804236}. The set-up is very similar to the undirected case, although we allow directed loops. We take $\vec G$ to be a strongly connected Eulerian directed graph. 
States are defined in the same way as for undirected graphs. Given a state $s$, a vertex other than $q$ can fire if $s(v) \geq d^+(v)$ where $d^+(v)$ is the outdegree of $v$.
The vertex $q$ can fire if and only if no other vertex can fire.
When a vertex $v$ fires, it sends one chip along each edge whose tail is $v$. Thus the state $s$ is transformed to a new state $s'$, so that 
\[ s'(u):=   \begin{cases} s(u)-d^+(u)+\Adj(G)_{u,u} & \text{if $u=v$,}\\
s(u) + \vec \Adj(G)_{v,u} & \text{otherwise.}
\end{cases}\]

A state $s$ is \emph{stable} if $s(v) < \deg^+(v)$ for all $v\in V-\{q\}$. A \emph{stabilization} $s^{\circ}$ of a state $s$ is a stable state reached from $s$ by firing vertices in $V-\{q\}$, possibly multiple times. It follows from~\cite[Lemma~2.4]{MR2477390} that every state has a unique stabilization.

A state $s$ is \emph{critical} if $s$ is stable and for every state $s_1$ there exists a state $s_2$ such that $(s_1+s_2)^{\circ}=s$. It follows from~\cite[Theorem~3]{zbMATH01256098} and~\cite[Theorem~3.11]{arxiv.1012.0287}  that a state $s$ is critical if and only if it is stable and there is a sequence of firings transforming $s$ to itself
in which each vertex fires once. 
The next result follows in a similar way to Theorem~\ref{ajkw}.

 \begin{theorem}
 \label{psndf}
Let $\vec G$ be a strongly connected Eulerian directed graph with $n$ vertices. Then every coset of $\langle \Delta(\vec G)\rangle$ in $\ints^n \cap \mathbf{1}^{\perp}$ contains precisely one critical state. 
It follows that we can define a group on the set of critical states of $\vec G$, by defining the sum of two critical states $s$ and $s'$ to be the unique critical state 
in the coset of $\langle \Delta(\vec G)\rangle$ in $\ints^n \cap \mathbf{1}^{\perp}$
containing $s+s'$ and that this group is isomorphic to $K(\vec G)$.
\end{theorem}

\medskip

Our attention returns to maps. Recall that our aim in this section is to describe the critical group $K(\bG)$ in terms of the critical states of a chip-firing game on $\bG$. That is, we want an analogue of Theorem~\ref{ajkw}.

By Theorem~\ref{thm:groupssame}, we have that $K(\bG)\cong (\ints^m\cap \mathbf{1}^{\perp})/\langle \Delta(\vec M(\bG))\rangle$ where $m$ is the number of edges of $\bG$. 
Applying Theorem~\ref{psndf} to the directed medial graph $\vec M(\bG)$ and translating  the chip-firing game to maps gives us our analogue, as follows. 

Let $\bG=[\sigma, \alpha, \phi]$ be a map with a distinguished edge $q$. 
A \emph{state} of $\bG$ is a function $s:E\rightarrow \mathbb{Z}$ such that
 $s(e)\geq 0$ for all $e\in E-\{q\}$ and $\sum_{e\in E} s(e)=0$. The value $s(e)$ gives the number of \emph{chips} on $e$. 
 An edge other than $q$ can \emph{fire}  if $s(e) \geq 2$.  
The edge $q$ can fire if and only if no other edge can fire.
When an edge $e$ \emph{fires}, two chips are removed from $e$ and are then replaced on the edges that contain the darts $\sigma(e^-)$ and $\sigma(e^+)$.

A state $s$ is \emph{critical} if and only if $s(e)<2$ for all $e\in E-\{q\}$ and there is a sequence of firings transforming $s$ to itself
in which each edge fires once.  
The following is then seen to hold as a consequence of Theorem~\ref{psndf}.

 \begin{theorem}\label{psndf1}
Let $\bG=[\sigma, \alpha, \phi]$ be a connected map with $m$ edges including a distinguished edge $q$. 
Then every coset of $\langle \Delta(\vec M(\bG))\rangle$ in $\ints^m \cap \mathbf{1}^{\perp}$ contains precisely one critical state. 
It follows that we can define a group on the set of critical states of $\vec G$, by defining the sum of two critical states $s$ and $s'$ to be the unique critical state 
in the coset of $\langle \Delta(\vec M(\bG))\rangle$ in $\ints^m \cap \mathbf{1}^{\perp}$
containing $s+s'$ and that this group is isomorphic to $K(\bG)$.
\end{theorem}

\begin{example}\label{ex:cf}
Let $\bG$ be the map in Figure~\ref{fex1}, and let $q=a$. If the initial state is 
\[(s(q),s(b),s(c),s(d))=(-1,0,1,0),\] then firing $q$ gives state $(-3,1,2,0)$, then firing $c$ gives $(-3,2,0,1)$, then firing $b$ gives $(-3,0,1,2)$, and finally firing $d$ gives $(-1,0,1,0)$. Thus   $(-1,0,1,0)$  is a  critical state. Similar calculations show that the set of critical states is
\[ \{(-1,0,1,0), (-1,1,0,0), (-2,0,1,1), (-2,1,0,1), (-2,1,1,0), (-3,1,1,1)\}.\]
\end{example}

\section{Further examples}
 \label{sec:examples}

\begin{figure}[!t]
    \centering
        \subfloat[ $K_5$ in the torus, giving $\bG_1$.]{
             \labellist
\small\hair 2pt
\pinlabel {$1$} at   31 47
\pinlabel {$2$} at   35 31
\pinlabel {$3$} at   50 35
\pinlabel {$4$} at    63 16
\pinlabel {$5$} at   39 11
\pinlabel {$6$} at    15 16 
\pinlabel {$7$} at   16 65
\pinlabel {$8$} at   64 62
\pinlabel {$9$} at   68 39 
\pinlabel {$9$} at  11 43
\pinlabel {$10$} at   45 50
\endlabellist
 \includegraphics[scale=1.5]{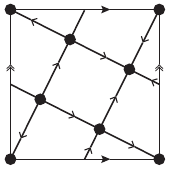} 
        \label{ex91a}}
        \qquad\qquad
         \subfloat[$K_{3,3}$ in the torus, giving $\bG_2$.]{
      \labellist
\small\hair 2pt
\pinlabel {$1$} at   53 60
\pinlabel {$2$} at  36 52
\pinlabel {$3$} at   28 45
\pinlabel {$4$} at    22 27
\pinlabel {$5$} at   56 31
\pinlabel {$6$} at    68 53
\pinlabel {$7$} at   58 10
\pinlabel {$8$} at   30 70
\pinlabel {$9$} at   10 13
\endlabellist
 \includegraphics[scale=1.5]{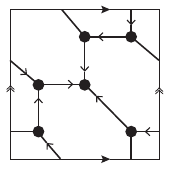} 
          \label{ex91b} 
        }

            \subfloat[The Petersen graph in the torus, giving $\bG_3$.]{
      \labellist
\small\hair 2pt
\pinlabel {$1$} at   45 62
\pinlabel {$2$} at   35 46
\pinlabel {$3$} at    32 36
\pinlabel {$4$} at   40 30
\pinlabel {$5$} at    50 36
\pinlabel {$6$} at     57 55
\pinlabel {$7$} at   25 56
\pinlabel {$8$} at   22 26
\pinlabel {$9$} at    60 26
\pinlabel {$10$} at   46 46
\pinlabel {$11$} at   48 11
\pinlabel {$12$} at   33 11
\pinlabel {$13$} at    10 68
\pinlabel {$14$} at   13 45
\pinlabel {$15$} at   9 15
\endlabellist
 \includegraphics[scale=1.5]{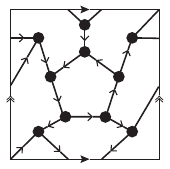} 
          \label{ex91c} 
        }
        \qquad\qquad
                  \subfloat[The Heawood graph in the torus, giving $\bG_4$.]{
      \labellist
\small\hair 2pt
\pinlabel {$1$} at   36 46
\pinlabel {$2$} at    45 40
\pinlabel {$3$} at    49 34
\pinlabel {$4$} at    47 20
\pinlabel {$5$} at     58 16
\pinlabel {$6$} at      57 71
\pinlabel {$7$} at    63 55
\pinlabel {$8$} at    67 50
\pinlabel {$9$} at     9 43
\pinlabel {$10$} at    17 31
\pinlabel {$11$} at    14 16
\pinlabel {$12$} at    26 9
\pinlabel {$13$} at     25 65
\pinlabel {$14$} at   37 60
\pinlabel {$15$} at   49 60
\pinlabel {$16$} at    59 39
\pinlabel {$17$} at    70 22
\pinlabel {$18$} at    11 60
\pinlabel {$19$} at     17 49
\pinlabel {$20$} at   27 28
\pinlabel {$21$} at   40 13
\endlabellist
 \includegraphics[scale=1.5]{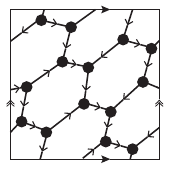} 
          \label{ex91d} 
        }
     \caption{Maps, drawn here as embedded graphs, considered in Section~\ref{sec:examples}. (The orientation is anticlockwise.)}
         \label{ex91}
\end{figure}

We close by presenting a collection of examples illustrating the calculation of the critical group of a map and contrasting it with the critical group of the underlying abstract graph.

For our first example, consider
Figure~\ref{ex91a} which shows a self-dual quadrangulation of the torus with edge set $E=\{1,\ldots,10\}$. Its underlying abstract graph is $K_5$, and the edges have been directed arbitrarily. 
The corresponding directed map $\vec{\bG}_1$ has 250 spanning quasi-trees of which 125 are spanning trees and the edge sets of the other 125 are the complements of the edge sets of spanning trees.
Taking the partial dual with respect to  $T_1=\{1,2,6,9\}$ 
 gives the bouquet  $\vec \bB_1=[\sigma,\alpha,\phi]$, where
 \[ \sigma = (1^-9^+8^+10^+3^+9^-6^+4^+8^-7^+6^-2^+5^+4^-3^-2^-1^+10^-5^-7^-).\]
The matrix $A(\vec{\mathbb{\bB}}_1)=A(\vec{\mathbb{\bG}}_1,T_1) $ is
 
\begin{equation}\label{exmatr2}
\begin{pmatrix} 
 0 & 0 & 0 & 0 &-1 & 0 &-1 & 0 & 0 &-1\\
 0 & 0 &-1 &-1 & 1 & 0 & 0 & 0 & 0 & 0\\
 0 & 1 & 0 & 0 & 1 & 0 & 1 &-1 &-1 & 0\\
 0 & 1 & 0 & 0 & 1 &-1 & 1 &-1 & 0 & 0\\
 1 &-1 &-1 &-1 & 0 & 0 & 0 & 0 & 0 &-1\\
 0 & 0 & 0 & 1 & 0 & 0 & 1 &-1 & 0 & 0\\
 1 & 0 &-1 &-1 & 0 &-1 & 0 & 0 & 0 &-1\\
 0 & 0 & 1 & 1 & 0 & 1 & 0 & 0 &-1 & 1\\
 0 & 0 & 1 & 0 & 0 & 0 & 0 & 1 & 0 & 1\\
 1 & 0 & 0 & 0 & 1 & 0 & 1 &-1 &-1 & 0
\end{pmatrix}.
\end{equation}
From Theorem~\ref{charpolycount} we find that the generating function for the number of spanning quasi-trees in $\mathbb{\bB}_1$ with respect to (two times the) genus is
\[ t^{10} P_{A(\vec{\mathbb{\bB}}_1)}(t^{-1}) = 16t^8 + 116t^6 + 96t^4 + 21t^2 + 1,\]  
and by computing the Smith Normal Form of $A(\vec{\mathbb{\bB}}_1)+I=A(\vec{\mathbb{\bG}}_1,T_1)+I $ we find that 
 the critical group $K(\bG_1)$ is $(\ints/5\ints)^2\oplus \ints/10\ints$.
Note that the classical critical group of the abstract graph $K_5$ is $(\ints/5\ints)^3$.

For our next example we consider the directed map $\vec \bG_2$  
shown as an embedded graph in Figure~\ref{ex91b}.  Its underlying abstract graph  is $K_{3,3}$. Taking $T_2=\{1,2,3,4,5\}$, we have
\[ A(\vec \bG_2,T_2)= \begin{pmatrix}
 1&0&0&0&0&1&-1&0&0\\
 0&1&0&0&0&1&-1&1&0\\
        0&0&1&0&0&-1&0&-1&1&\\
        0&0&0&1&0&0&0&-1&1&\\
        0&0&0&0&1&0&1&0&-1&\\
        -1&-1&1&0&0&1&-1&1&0&\\
        1&1&0&0&-1&1&1&0&-1&\\
        0&-1&1&1&0&-1&0&1&1&\\
        0&0&-1&-1&1&0&1&-1&1
   \end{pmatrix}.\]
From this we find that the critical group $K(\bG_2)$ is 
$\ints/6\ints\oplus \ints/18\ints$.
For comparison, the  classical critical group $K(K_{3,3})$  of its underlying abstract graph  is $(\ints/3\ints)^2  \oplus \ints / 9\ints$.

Next consider the unique embedding of the Petersen graph in the torus shown in Figure~\ref{ex91c}, and giving a map $\bG_3$. Taking $T_3=\{1,2,3,4,5,6,7,8,9\}$, gives
\[ A(\vec \bG_3,T_3)= 
\begin{pmatrix}
 1&0&0& 0&0&0& 0&0&0& 0&-1&-1&    0&0&0\\
 0&1&0& 0&0&0& 0&0&0& -1&-1&-1&    0&0&0\\
 0&0&1& 0&0&0& 0&0&0& -1&-1&-1&   -1&-1&0\\
0&0&0&  1&0&0& 0&0&0& -1&-1&0&    -1&-1&-1\\
0&0&0&  0&1&0& 0&0&0& -1&0&0&    -1&0&-1\\
0&0&0&  0&0&1& 0&0&0& 0&0&0&    -1&0&-1\\
0&0&0&  0&0&0& 1&0&0& 0&0&0&    -1&-1&0\\
0&0&0&  0&0&0& 0&1&0& 0&0&-1&    0&0&1\\
0&0&0&  0&0&0& 0&0&1& 0&-1&0&    0&-1&0\\
0&1&1& 1&1&0&  0&0&0& 1&0&0&     0&0&0\\
1&1&1& 1&0&0&  0&0&1& 0&1&0&     1&1&1\\
1&1&1& 0&0&0&  0&1&0& 0&0&1&     1&1&1\\
0&0&1& 1&1&1&  1&0&0& 0&-1&-1&  1&0&1\\
0&0&1& 1&0&0&  1&0&1& 0&-1&-1&  0&1&1\\
0&0&0& 1&1&1&  0&-1&0& 0&-1&-1&  -1&-1&1   
   \end{pmatrix}.\]
Then we find that the critical group is 
$\ints/2\ints\oplus \ints/1270\ints$.
For comparison, the classical critical group of the Petersen graph (as an abstract graph) is $\ints/2\ints\oplus(\ints/10\ints)^3$.

\medskip


We now consider a computational technique arising from  Theorem~\ref{thm:pet}. In general, given  a directed connected map $\vec \bG$ and the edge set $T$ of one of its spanning quasi-trees, computing the matrix $A(\vec \bG,T)+I$ is not difficult. Nevertheless, the matrices $B$ and $D'$ of Section~\ref{sec:connections} are more manageable and $B^t B+D'+I$ is more succinct in general. Thus we can use them to streamline computations of the critical group.

To illustrate this we reconsider the critical group of the map $\bG_1$ given by Figure~\ref{ex91a}, and the matrix $A(\vec \bG_1,T_1)$ in~\eqref{exmatr2}.
The matrix $B$ is then the submatrix of $A(\vec \bG_1,T_1)$  with row labels in 
  $T_1$ and column labels in $E-T_1$, and $D'$ is the principal submatrix with labels in $E-T_1$. Then 
  \[ B^t B+D'+I = \begin{pmatrix} 3&1&0&1&0&1\\1&3&0&2&-2&0\\-2&-2&3&1&0&0\\
  -1&0&1&3&-1&0\\2&0&0&-1&3&2\\1&0&2&2&0&3\end{pmatrix}.\]
Computing the Smith Normal Form of this matrix once again shows that $K(\bG_1)=(\ints/5\ints)^2\oplus \ints/10\ints$.

As another example of this method, 
consider Figure~\ref{ex91d} which shows 
the unique embedding of the Heawood graph in the torus. Its edges have been directed arbitrarily and we let   $\vec \bG_4$ denote the corresponding directed map.
 The underlying abstract graph of its dual $\bG_4^*$ is $K_7$. The set $T_4=\{1,2,\ldots,13\}$ is the edge set of a spanning tree of $\bG_4$.
 
To compute $K(\bG_4)$, again using the techniques suggested by Theorem~\ref{thm:pet}, we have that the information in the $21\times 21$ matrix $A(\vec \bG_4,T_4)+I$ may be compressed into the matrix $B^t B+I+D'$ 
 \[
\begin{pmatrix} 
14& 4 &4 &4 & 4& 10 &10 & 10 \\
6& 6 &3 &1 & 0& 6 & 5 & 3 \\
6& 3 &6 &3 & 1& 6 & 6 & 5 \\
6& 1 &3 &6 & 3& 5 & 6 & 6 \\
6& 0 &1 &3 & 6& 3 & 5 & 6 \\
8& 4 &4 &3 & 1& 10 &7 & 5 \\
8& 3 &4 &4 & 3& 7 & 10 & 7 \\
8& 1 &3 &4 & 4& 5 & 7 & 10 \\
\end{pmatrix}.
\]
By considering the Smith Normal Form of this matrix we find critical group $K(\bG_4)$ of the  Heawood graph embedded in the torus (and also of its dual, an embedding of $K_7$ in the torus)  is $(\ints/7\ints)^3\oplus(\ints/14\ints)^2$.
The classical critical group of the Heawood graph as an abstract graph is $(\ints/7\ints)^4\oplus(\ints/21\ints)$.

 We continue with an observation on this computational technique. For a map with $m$ edges and with a spanning tree with $r$ edges, the matrix $D'$ has $m-r$ rows and columns. As the critical group of a map $\bG$ is unchanged by taking partial duals, a strategy to minimize the size of the matrices involved in its computation is to find a partial dual of $\bG$ with the maximum possible number of vertices. But this strategy has limitations
 and does not always lead to a matrix $D'$ significantly smaller than $A(\vec \bG,T)+I$.
 For example, consider the bouquet $\mathbb{N}_m$, comprising $m$ edges which are pairwise interlaced. Each of its partial duals has at most two vertices.  
  After appropriately directing the edges, the matrix 
 $A(\vec{\mathbb{N}}_m)+I$  
 has $1$ in every position on and above the main diagonal and $-1$ in every position below the main diagonal. 
 Basic matrix manipulation shows that the invariant 
 factors are a single one and $m-1$ twos. Thus, the number of spanning quasi-trees in 
 $\mathbb{N}_m$ is $2^{m-1}$ and the critical group 
 $K(\mathbb{N}_m)$ is $(\ints/2\ints)^{m-1}$. Furthermore, as there are $m-1$ invariant factors the critical group cannot be computed from a matrix with fewer than $m-1$ rows and columns.

 \medskip  
 
Finally, we conclude with an open question on the form of the critical groups of maps. For any finite abelian group $H:=\bigoplus_{i=1}^{s} \mathbb{Z}/p_i\mathbb{Z}$ 
 there is a planar abstract graph $G$ with classical critical group $H$. Thus, any finite abelian group is the critical group 
 of some plane map $\bG$. It is, however, not difficult to prove that there is no $2$-connected 
 abstract graph $G$  with critical group $K(G)$ isomorphic to  $\ints/2\ints\oplus\ints/
 2\ints$. But, the bouquet $\mathbb{N}_3$ comprising $3$ pairwise interlaced edges is non-separable (meaning here that it does not arise as the map amalgamation, or connected sum, of two non-trivial maps) 
 and has critical group isomorphic to $\ints/2\ints\oplus\ints/
 2\ints$.
 Thus, an interesting question is to determine whether 
 every finite abelian group is the critical group of a non-separable map.

\bibliographystyle{alphaurl}
\bibliography{chip}

\end{document}